\documentclass[11pt]{amsart}

\usepackage{graphicx}
\usepackage{framed}

\usepackage{color}

\usepackage[utf8]{inputenc}

\usepackage[all]{xy}
\usepackage{pb-diagram, pb-xy}

\usepackage{amssymb}
\usepackage{amsthm}

\usepackage{tikz}
\usepackage{pgffor}

\usepackage{hyperref}

\usepackage{calc}

\def\part{\@startsection{part}{0}%
\z@{\linespacing\@plus\linespacing}{.5\linespacing}%
{\large\bfseries\centering}}

\newtheoremstyle{saetze} 
    {5pt}                    
    {5pt}                    
    {\itshape}                   
    {12pt}                           
    {\bfseries}                   
    {.}                          
    {.5em}                       
    {}  

\theoremstyle{saetze}
\newtheorem{thm}{Theorem}[section]
\newtheorem{lem}[thm]{Lemma}
\newtheorem{cor}[thm]{Corollary}
\newtheorem{prop}[thm]{Proposition}

\newtheorem{conj}[thm]{Conjecture} 
\theoremstyle{definition}

\newtheorem{example}[thm]{Example}
\newtheorem{remark}[thm]{Remark}

\newcommand{\calN}{{\mathcal N}}

\newcommand{\calR}{{\mathcal R}}

\newcommand{\one}{{\mathbf{1}}}

\newcommand{\Z}{{\mathbf{Z}}}

\def\N{\mathbb{N}}
\newcommand{\A}{{\mathbb A}}

\newcommand{\lift}{\operatorname{lift}}

\usepackage{xypic,dsfont}

\hypersetup{
    colorlinks,
    citecolor=red,
    filecolor=black,
    linkcolor=blue,
    urlcolor=black
}

\DeclareRobustCommand{\gobblefive}[5]{}

\makeatletter
\def\part{\@startsection{part}{1}%
\z@{.7\linespacing\@plus\linespacing}{.5\linespacing}%
{\large\scshape\centering}}
\makeatother

\begin{document}
 \title{Pieri type rules and $GL(2|2)$ tensor products}

\author{Thorsten Heidersdorf} 
\address{T.H.: Department of Mathematics, The Ohio State University}
\address{T.H.: Max-Planck Institut f\"ur Mathematik, Bonn}
\address{T.H.: Institut f\"ur Mathematik, Universit\"at Bonn}
\email{heidersdorf.thorsten@gmail.com} 

\author{Rainer Weissauer}
\address{R.W.: Mathematisches Institut, Universit\"at Heidelberg}
\email{weissauer@mathi.uni-heidelberg.de}

\date{}

\begin{abstract} We derive a closed formula for the tensor product of a family of mixed tensors using Deligne's interpolating category $\underline{Rep}(GL_{0})$. We use this formula to compute the tensor product of a family of irreducible $GL(n|n)$-representations. This includes the tensor product of any two maximal atypical irreducible representations of $GL(2|2)$.
\end{abstract}



\thanks{{T.H.: Corresponding author}}

\maketitle

\section{Introduction}

For the classical group $GL(n)$ the tensor product decomposition \[L(\lambda) \otimes L(\mu) = \bigoplus_{\nu} c_{\lambda \mu}^{\nu} L(\nu)\] between two irreducible representations is given by the Littlewood-Richardson rule for the Littlewood-Richardson coefficients $c_{\lambda \mu}^{\nu}$. Contrary to this case the analogous decomposition between two irreducible representation of the General Linear Supergroup $GL(m|n)$ is poorly understood. A classical result from Berele and Regev \cite{Berele-Regev} and Sergeev \cite{Sergeev} shows that the fusion rule between direct summands of tensor powers $V^{\otimes r}$ of the standard representation $V \simeq k^{m|n}$ is again given by the Littlewood-Richardson rule. The first more general results were achieved in \cite{Heidersdorf-mixed-tensors} where we obtained a decomposition law for tensor products between any two mixed tensors, direct summands in a mixed tensor space $V^{\otimes r} \otimes (V^{\vee})^{\otimes s}$, $r,s \in \N$. This result is based on the tensor product decomposition in Deligne's interpolating category $\underline{Rep}(GL_{\delta})$ \cite{Deligne-interpolation}. Due to the universal property of Deligne's category, we have for $\delta = m-n$ a tensor functor $F_{m|n}: \underline{Rep}(GL_{m-n}) \to Rep(GL(m|n))$ sending the standard representation of the Deligne category to the standard representation $V = k^{m|n}$ of $GL(m|n)$. Since the decomposition of the tensor product of two indecomposable elements is known for $\underline{Rep}(GL_{m-n})$ by results of Comes and Wilson \cite{Comes-Wilson}, we obtain an analogous decomposition law once we describe the image $F_{m|n}(X)$ of an arbitrary indecomposable object $X$ in $\underline{Rep}(GL_{m-n})$. This was achieved in \cite{Heidersdorf-mixed-tensors} based on results by Brundan and Stroppel \cite{Brundan-Stroppel-5} on the interplay between Khovanov algebras and Walled Brauer algebras. Since any Kostant module \cite{Brundan-Stroppel-2} and any projective representation is a mixed tensor (up to some Berezin twist) \cite{Heidersdorf-mixed-tensors}, these results give a decomposition law for their tensor products, covering in particular the decomposition between any two irreducible $GL(m|1)$-representations.

\subsection{The main results} For $m,n \geq 2$ the irreducible mixed tensors are rather special. For example no non-trivial maximal atypical irreducible representation of $GL(n|n)$ is a mixed tensor. It is well-known that the weight of a maximal atypical representation is of the form $\lambda = (\lambda_1,\ldots,\lambda_n \ | \ -\lambda_n,\ldots,-\lambda_1)$, and we denote the corresponding irreducible representation by $[\lambda_1,\ldots,\lambda_n]$. We also denote the irreducible representation $[i,0,\ldots,0]$ by $S^i$ for $i \geq 0$. In this paper we obtain an almost complete picture for the tensor product $S^i \otimes S^j$ for any $i,j$ and any $n$ and show \[ S^i \otimes S^j \cong \delta_{ij}Ber^{i-1} \oplus M_{ij} \oplus \text{ semisimple part} \] where $M_{ij}$ is a maximal atypical indecomposable representation and where the semisimple part is of atypicality $n-2$ and can be explicitely understood in terms of the decomposition law for $G_0$. This is the only known case of a formula apart from the $GL(m|1)$-case and the case of Kostant modules. 

\medskip
Our interest in this result comes from different sources. 
\begin{enumerate} 
\item In recent years the structure of $\mathcal{T}_n$ as an abelian category was determined in \cite{Brundan-Stroppel-4}. While other questions remain (e.g. an analogue of Borel-Weil-Bott, a more convenient character formula etc.), we view the description of the monoidal properties of $\mathcal{T}_n$ or its analogues for the other supergroups as one of the central questions in the theory. We hope that these results shed some light on this difficult and important problem.
\item The Lie superalgebra $\mathfrak{gl}(2|2)$ or its simple counterpart $\mathfrak{psl}(2|2)$ occurs in several physical models of AdS theory \cite{Matsumoto-Molev} \cite{Quella-Schomerus}. In fact similar formulas for the fusion rules have been obtained before in the more restrictive $\mathfrak{psl}(2|2)$-case in the physics literature \cite{Goetz-Quella-Schomerus-psl}.
\item The quotient of $\mathcal{T}_n = Rep(GL(n|n))$ by its largest proper tensor ideal $\mathcal{N}$ is the representation category of a supergroup scheme \cite{Heidersdorf-semisimple} \cite{Heidersdorf-Weissauer-Tannaka}. The fusion rules obtained in this paper for $S^i \otimes S^j$ for $n=2$ play a crucial role in the determination of this group. More precisely, for every irreducible representation $L(\lambda)$ of non-vanishing superdimension we consider the tensor subcategory generated by it in $\mathcal{T}_n/\mathcal{N}$. It is equivalent to the representation category of a classical group $H_{\lambda}$ \cite{Heidersdorf-Weissauer-Tannaka}. In fact these groups (or their connected derived groups) can be understood inductively starting with the case $n=2$ considered in this paper. The inductive determination rests then on the super tannakian formalism of Deligne and dimension and rank estimates. In this sense the case $n=2$ is harder than the higher rank cases for $n \geq 3$. 
\end{enumerate}



\subsection{Summary of the proof} While the article uses a fair amount of computation, its approach is rather conceptual and uses a lot of theory: Deligne's interpolating categories $\underline{Rep}(GL_{m-n})$, the description of the functor $F_n: \underline{Rep}(GL_0) \to Rep(GL(n|n))$, the knowledge of the Loewy layers of mixed tensors based on Brundan and Stroppels results about the connection between the walled Brauer algebra and Khovanov algebras \cite{Brundan-Stroppel-5} and last but not least the formalism of cohomological tensor functors and Tannaka groups of \cite{Heidersdorf-Weissauer-tensor} \cite{Heidersdorf-Weissauer-Tannaka}. Let us describe the necessary steps.

\subsubsection{Step 1: The Deligne category $\underline{Rep}(GL_{\delta})$}
A first input is the tensor product decomposition in the Deligne category $\underline{Rep}(GL_0)$ \cite{Deligne-interpolation} \cite{Comes-Wilson}. The information about the tensor product decomposition between indecomposable objects in $\underline{Rep}(GL_0)$ can be transferred to  $\mathcal{T}_n$ by means of the symmetric monoidal functor $F_{n|n}: \underline{Rep}(GL_{0}) \to Rep(GL(n|n))$ sending the standard representation of the Deligne category to the standard representation $V = k^{n|n}$ of $GL(n|n)$. This functor can be described explicitely as in \cite{Comes-Wilson} \cite{Heidersdorf-mixed-tensors}. Implicitely results about the representation theory of the walled Brauer algebra play a crucial role here.

\subsubsection{Step 2: Khovanov algebras} \label{step-2}
The image $F_{n|n}(R(\lambda))$ of an indecomposable object $R(\lambda) \in \underline{Rep}(GL_0)$ can be determined \cite{Heidersdorf-mixed-tensors} based on a description of the Loewy structure of direct summands in a mixed tensor space $V^{\otimes r} \otimes (V^{\vee})^{\otimes s}$ \cite{Brundan-Stroppel-5}.   The results of Brundan and Stroppel are based on the combinatorial description of the blocks in $\mathcal{R}_n$ by means of the diagram algebra $K(m|n)$, a so-called Khovanov algebra \cite{Brundan-Stroppel-4}. The Loewy layers and the composition factors of the $F_{n|n}(R(\lambda))$ admit then a description in terms of the cup/cap combinatoric of the Khovanov algebras or, in other words, parabolic Kahzdan-Luztig theory for a maximal parabolic in type $A$.

\subsubsection{Step 3: The modules $\A_{S^i}$}
None of the irreducible representations $S^i$ is of the form $F_{n|n}(R(\lambda))$ (except for the trivial case $S^0 \cong \one$). To circumvent this problem, we first consider special mixed tensors $\A_{S^i}$ which contain the irreducible representation $S^i$ as the constituent of highest weight with multiplicity 1. As all mixed tensors these are of the form $F_{n|n}(R(\lambda))$ \cite{Comes-Wilson}. We proceed as follows: We derive a closed formula for the tensor product decomposition $\A_{S^i} \otimes \A_{S^j}$. This computation takes place in $\underline{Rep}(GL_0)$ and is then pushed to $\mathcal{T}_n$ via $F_{n|n}$.

\subsubsection{Step 4: $K_0$-decomposition and the case $n=2$}
The composition factors and the socle filtration of the $\A_{S^i}$ are known \cite{Heidersdorf-mixed-tensors}. The tensor product $\A_{S^i} \otimes \A_{S^j}$ splits into representations of the form $\A_{S^{k}}$ for some $k$, maximal atypical mixed tensors $R(a,b)$ and a semisimple part which is not maximal atypical. We specialize now to the case $n=2$ and view this decomposition in the Grothendieck ring $K_0(\mathcal{T}_2)$ and derive from this a closed formula for $S^i \otimes S^j$ in the Grothendieck ring. The difficult part here is to understand the maximal atypical part since the other summands of lower atypicality are irreducible. In fact we can easily derive a general formula for the non maximal atypical part in $S^i \otimes S^j$ for any $n$ as in section \ref{sec:gl(2|2)-lower-atypicality}: The remaining composition factors in $\A_{S^i} \otimes \A_{S^j}$ are all $(n-2)$-fold atypical and it is easy to see that they always lie in different blocks. Hence they cannot combine to an indecomposable representation and the $K_0$-decomposition is already enough for the computation. The reason for the specialization to $n=2$ is that we need an explicit description of the composition factors of the maximal atypical indecomposable mixed tensors $R(a,b)$. While this is theoretically possible by step \ref{step-2}, the combinatorics becomes very complex for $n \geq 3$. In the $n=2$ these modules are projective covers and their composition factors can be described easily.

We split the computation of $S^i \otimes S^j$ into two parts. We first project onto the maximal atypical block $\Gamma$ and then compute the remaining summands afterwards in section \ref{sec:gl(2|2)-lower-atypicality}. In \[ [ \A_{S^i}] \otimes [\A_{S^j}] \in K_0(\mathcal{R}_2) \] the tensor product $S^i \otimes S^j$ occurs exactly once, and all other tensor products are of the form $Ber^{p} (S^k \otimes S^l)$ with both $k$ and $l$ less or equal to $i$ and $j$ and some Berezin power $p$. This allows us to compute the maximal atypical composition factors of $S^i \otimes S^j$ recursively in lemma \ref{composition-factors}.

\subsubsection{Step 5: Cohomological tensor functors for $n=2$}
In order to determine the decomposition of $S^i \otimes S^j$ into maximal atypical indecomposable representations we use the theory of cohomological tensor functors \cite{Heidersdorf-Weissauer-tensor}. Here we consider the tensor functor $DS: Rep(GL(2|2)) \to Rep(GL(1|1))$. The main theorem of \cite{Heidersdorf-Weissauer-tensor} gives a formula for $DS(L)$ for any irreducible representation and we get $DS(S^i) = Ber^i \oplus \Pi^{1-i} Ber^{-1}$ where $\Pi$ denotes the parity shift functor. This gives us strict estimates on the number of indecomposable summands and their superdimension which is enough to determine the indecomposable summands in theorem \ref{gl-2-2-final}. 

\medskip
The steps 1 - 5 settle the entire $n=2$ case as well as the computation of the summands which are not maximal atypical for any $n$. The general case for $n \geq 3$ now follows rather easily assuming the formalism developed in \cite{Heidersdorf-Weissauer-Tannaka}. Therefore some results in section \ref{reduction} depend on \cite{Heidersdorf-Weissauer-Tannaka}. 

\subsubsection{Step 6: Reduction to the $n=2$-case}
We show in lemma \ref{thm:clean} that the $S^i \otimes S^j$ decomposition is always clean: every maximal atypical summand has non-vanishing superdimension. This means that we can the see the decomposition behaviour in the quotient of $\mathcal{T}_n$ by the ideal of negligible morphisms. But the structure of this quotient has been exactly determined in \cite{Heidersdorf-Weissauer-Tannaka} using the fusion rules for $n=2$! We stress that we need very little from the general setup of \cite{Heidersdorf-Weissauer-Tannaka} to deal with the $S^i$-case. In fact the $S^i$ case can be completely separated from the remaining cases as shown in \cite[Section 9.3]{Heidersdorf-Weissauer-Tannaka}.

\medskip
These methods allow in principle to compute also the maximal atypical composition factors in the decomposition $S^i \otimes S^j$ for $n \geq 3$. However it is difficult to determine the composition factors of the maximal atypical mixed tensors $R(a,b)$ for $n \geq 3$. We end the article with a conjecture for the socle of $S^i \otimes S^j$ for arbitrary $n$.

\section{ The superlinear groups}\label{2}

\subsection{Representations} Let $k$ be an algebraically closed field of characteristic zero. Let $\mathfrak{g} = \mathfrak{gl}(n|n) = \mathfrak{g}_0 \oplus \mathfrak{g}_1$ be the general linear superalgebra and $GL(n|n)$ the general linear supergroup. By definition a finite dimensional super representation $\rho$ of $\mathfrak{gl}(n\vert n)$ defines a representation $\rho$ of $GL(n \vert n)$ if its restriction to $\mathfrak{g}_0$ comes from an algebraic representation of $G_0 = GL(n) \times GL(n)$, also denoted $\rho$. We denote the category of finite-dimensional representations with parity-preserving morphisms by $\mathcal{T}_n = \mathcal{T}_{n|n}$. For $M \in \mathcal{T}$ we denote by $M^{\vee}$ the ordinary dual and by $M^*$ the twisted dual. For simple and for projective objects $M$ of $\mathcal{T}$ we have $M^*\cong M$ \cite{Brundan-Kazhdan}.

\subsection{The category ${\calR}_n$} \label{sec:rep} Fix the morphism $\varepsilon: \mathbb Z/2\mathbb Z \to G_{0}=GL(n)\times GL(n)$ which maps $-1$ to the element 
$diag(E_n,-E_n)\in GL(n)\times GL(n)$ denoted $\epsilon_{nn}$. We write $\epsilon_n = \epsilon_{nn}$. Note that
$Ad(\epsilon_{nn})$ induces the parity morphism on the Lie superalgebra $\mathfrak{gl}(n|n)$ of $G$. 
We define the abelian subcategory
$\calR_n$ of $\mathcal{T}_n$ as the full subcategory of all objects $(V,\rho)$ in $\mathcal{T}_n$  
with the property $  p_V = \rho(\epsilon_{nn})$; here $\rho$ denotes the underlying homomorphism $\rho: GL(n)\times GL(n) 
\to GL(V)$ of algebraic groups over $k$ and $p_V$ the parity automorphism of $V$.
The subcategory ${\calR}_n$ is stable under the dualities ${}^\vee$ and $^*$. The irreducible representations in $\mathcal{R}_n$ are indexed by dominant integral weights with respect to the standard Borel subalgebra of upper triangular matrices. We denote by $L(\lambda)$ the irreducible representation with highest weight $\lambda = (\lambda_1,\ldots, \lambda_n | \lambda_{n+1}, \ldots, \lambda_{2n})$ where $\lambda \in \Z^{2n}$ is any element satisfying \[ \lambda_1 \geq \lambda_2 \geq \ldots \geq \lambda_n \text{ and } \lambda_{n+1} \geq \lambda_{n+2} \geq \ldots \geq \lambda_{2n}.\] The Berezin determinant of $GL(n|n)$
defines a one dimensional representation $B=Ber$ with weight $(1,\ldots,1 \ | \ -1, \ldots, -1)$. For each representation $M \in \mathcal{R}_n$ we also have its parity shifted version $\Pi(M)$ in $\mathcal{T}_n$. Since we only consider parity preserving morphisms, these two are not isomorphic. In particular the irreducible representations in $\mathcal{T}_{n}$ are given by the $\{L(\lambda), \Pi L(\lambda) \ | \ \lambda \in X^+ \}$. The whole category $\mathcal{T}_n$ decomposes as  $\mathcal{T}_{n} = \calR_{n} \oplus \Pi \calR_{n}$ \cite[Corollary 4.44]{Brundan-Kazhdan}. An object $M\in \mathcal{T}_n$ is called negligible, if it is the direct
sum of indecomposable objects $M_i$ in $\mathcal{T}_n$ with superdimensions $sdim(M_i)=0$. The thick ideal of negligible objects is denotes $\calN$ or $\calN_n$.

\subsection{Atypicality.} If $L(\lambda)$ is projective, the weight $\lambda$ is called typical. If not, $\lambda$ is called atypical. The atypicality of a weight can be measured by a number between $0$ and $n$ \cite{Kac-Rep}. If the atypicality is $n$, we say the weight is maximal atypical. An example is the Berezin determinant  $Ber$ of dimension $1$. More generally an irreducible representation is maximal atypical if and only if $\lambda$ is of the form \[ \lambda = (\lambda_1,\ldots,\lambda_n \ | \ -\lambda_n,\ldots,-\lambda_1).\] In this case we often write $[\lambda_1,\ldots,\lambda_n]$ for $L(\lambda)$. We define for $i \geq 0$ \[ S^i = [i,0,\ldots,0]  \ \in \mathcal{R}_n.\] The superdimension of an irreducible representation is non-zero if and only if $L(\lambda)$ is  maximal atypical \cite{Serganova-kw} \cite{Weissauer-gl}.The abelian categories $\mathcal{T}_{n}$ and $\calR_n$ decompose into blocks and the degree of atypicality is a block-invariant.

\section{Mixed tensors} \label{stable0}

The decompositon of $S^i \otimes S^j$ is obtained from the decomposition $\A_{S^i} \otimes \A_{S^j}$ where the $\A_{S^i}$ are mixed tensors. We review some facts about them.

\subsection{Indecomposable representations and combinatorics of bipartitions} Let $\mathcal{MT}$ denote the full subcategory of mixed tensors in ${\calR_n}$ whose objects are direct sums of the indecomposable objects in ${\calR}_n$ that appear in a decomposition $V^{\otimes r} \otimes (V^{\vee})^{\otimes s}$ for some natural numbers $r,s \geq 0$, where $V \in {\calR_n}$ denotes the standard representation. By \cite[Theorem 8.19]{Brundan-Stroppel-5} and \cite[Theorem 8.18]{Comes-Wilson} the indecomposable objects in $\mathcal{MT}$ are parametrized by $(n|n)$-cross bipartitions (see below). Let $R_n(\lambda)$ (or $R(\lambda)$ if the dependency on $n$ is clear) denote the indecomposable representation in ${\calR}_n$ corresponding to the bipartition $\lambda = (\lambda^L ,\lambda^R)$ under this parametrization. We sometimes write $R(\lambda^L, \lambda^R)$ to avoid brackets. To any bipartition we attach a weight diagram in the sense of \cite{Brundan-Stroppel-1}, i.e. a labelling of the numberline $\mathbb Z$ according to the following dictionary. Put \[ I_{\wedge}(\lambda)  := \{ \lambda_1^L, \lambda_2^L - 1, \lambda_3^L - 2, \ldots \} \quad \text{and}\quad  I_{\vee}(\lambda)  := \{1  -\lambda_1^R, 2 - \lambda_2^R, \ldots \}\ . \] Now label the integer vertices $i$ on the numberline by the symbols $\wedge, \vee, \circ, \times$ according to the rule \[ \begin{cases} \circ \quad \text{ if } \ i \  \notin I_{\wedge} \cup I_{\vee}, \\ \wedge \quad \text{ if } \ i \in I_{\wedge}, \ i \notin I_{\vee}, \\ \vee \quad \text{ if } \ i \in I_{\vee}, \ i \notin I_{\wedge}, \\ \times \quad \text{ if } \ i \in I_{\wedge} \cap I_{\vee}. \end{cases} \]  To any such data one attaches a cup-diagram as in \cite[6.3]{Comes-Wilson} or \cite{Brundan-Stroppel-1} and we define the following three invariants \begin{align*} rk(\lambda) & = \text{ number of crosses } \\ d(\lambda) & = \text{ number of cups } \\ k(\lambda) & = rk(\lambda) + d(\lambda). \end{align*} 

\begin{example} \label{example-cup} Let $\lambda = (3,1^3)$. Then its weight diagram is 
\medskip

\begin{center}
\begin{tikzpicture}
 \draw (-6,0) -- (6,0);
\foreach \x in {0,1,2,4,5} 
     \draw (\x-.1, .2) -- (\x,0) -- (\x +.1, .2);
\foreach \x in {-1,-2,-3,-4,3} 
     \draw (\x-.1, -.2) -- (\x,0) -- (\x +.1, -.2);
\foreach \x in {} 
     \draw (\x-.1, .1) -- (\x +.1, -.1) (\x-.1, -.1) -- (\x +.1, .1);

\end{tikzpicture}
\end{center}
\medskip

with the rightmost $\wedge$ at the vertex 3. Its cup diagram is 

\medskip
\begin{center}
  \begin{tikzpicture}
 \draw (-6,0) -- (6,0);
\foreach \x in {} 
     \draw (\x-.1, .2) -- (\x,0) -- (\x +.1, .2);
\foreach \x in {} 
     \draw (\x-.1, -.2) -- (\x,0) -- (\x +.1, -.2);
\foreach \x in {} 
     \draw (\x-.1, .1) -- (\x +.1, -.1) (\x-.1, -.1) -- (\x +.1, .1);

\draw [-,black,out=270,in=270](2,0) to (3,0);

\end{tikzpicture}
\end{center}
\medskip

Therefore $k(\lambda) =1$, $rk(\lambda) = 0$ and $d(\lambda) = 1$. For $\lambda = ((4,2), (2^2, 1^{2}))$ the weight diagram is

\medskip

\begin{center}
\begin{tikzpicture}
 \draw (-6,0) -- (6,0);
\foreach \x in {-1,0,2,3,5} 
     \draw (\x-.1, .2) -- (\x,0) -- (\x +.1, .2);
\foreach \x in {4,1,-2,-3,-4} 
     \draw (\x-.1, -.2) -- (\x,0) -- (\x +.1, -.2);
\foreach \x in {} 
     \draw (\x-.1, .1) -- (\x +.1, -.1) (\x-.1, -.1) -- (\x +.1, .1);

\end{tikzpicture}
\end{center}
\medskip

with the two rightmost $\wedge$'s at the vertices $4$ and $0$. Its cup diagram is 

\medskip
\begin{center}
  \begin{tikzpicture}
 \draw (-6,0) -- (6,0);
\foreach \x in {} 
     \draw (\x-.1, .2) -- (\x,0) -- (\x +.1, .2);
\foreach \x in {} 
     \draw (\x-.1, -.2) -- (\x,0) -- (\x +.1, -.2);
\foreach \x in {} 
     \draw (\x-.1, .1) -- (\x +.1, -.1) (\x-.1, -.1) -- (\x +.1, .1);

\draw [-,black,out=270,in=270](3,0) to (4,0);
\draw [-,black,out=270,in=270](0,0) to (1,0);

\end{tikzpicture}
\end{center}
\medskip

Therefore $k(\lambda) =2$, $rk(\lambda) = 0$ and $d(\lambda) = 2$.
\end{example}

A bipartition is said to be $(n|n)$-cross if and only if $k(\lambda) \leq n$. By \cite[Lemma 8.18]{Brundan-Stroppel-5} the modules $R( \lambda^L, \lambda^R)$ have irreducible socle and cosocle equal to $L(\lambda^{\dagger})$ where the highest weight $\lambda^{\dagger}$ can be obtained by a combinatorial algorithm from $\lambda$. Let $\theta: \Lambda \to X^+(n)$ denote the resulting map $\lambda \mapsto \lambda^\dagger$ between the set of $(n|n)$-cross bipartitions $\Lambda$ and the set $X^+(n)$ of highest weights of $\calR_n$.

\begin{thm} \cite[Corollary 5.4, Theorem 5.12]{Heidersdorf-mixed-tensors} \label{mixed-structure} $R = R(\lambda^L,\lambda^R)$ is an indecomposable module of Loewy length $2 d(\lambda) + 1$. It is projective if and only if $k(\lambda) = n$ in which case we have $R = P(\lambda^{\dagger})$. In particular $R(\lambda)$ is irreducible if $d(\lambda) = 0$.
\end{thm}

\subsection{The map $\lambda \mapsto \lambda^{\dagger}$} \label{sec:theta} We recall the explicit description of the map $\theta$ as in \cite[Section 6.1]{Heidersdorf-mixed-tensors}, i.e. we describe how to transform the weight diagram of the bipartition $\lambda$ into the weight diagram of the highest weight $\lambda^{\dagger}$. Define $M$ to be the largest vertex labelled with a $\times$ or $\circ$ or part of a cup in the weight diagram of $\lambda$ and put \[ T= max(k(\lambda) + 1, M+1).\]  Now define \begin{align*} X= \begin{cases} 0 & M+1 \leq k(\lambda) + 1 \\ M- k(\lambda) & \text{else.} \end{cases} \end{align*} We say a vertex is free if it does not have a cross, or a circle or is not part of a cup. 

\begin{cor} (\cite[Corollary 6.3]{Heidersdorf-mixed-tensors}) The weight diagram of $\lambda^{\dagger}$ is obtained from the weight diagram of $\lambda$ by switching all $\vee$'s to $\wedge$'s and vice versa at vertices $\geq T$ and switching all $\vee$'s to $\wedge$'s and vice versa at the first $X + n-k(\lambda)$ free vertices $< T$. 
\end{cor}


\subsection{Deligne's interpolating category} For every $\delta \in k$ we denote by $\underline{Rep}(GL_{\delta})$ the interpolating category defined in  \cite{Deligne-interpolation}. This is a $k$-linear pseudoabelian rigid symmetric monoidal category. By construction it contains an object $st$ of dimension $\delta$, called the standard representation. By the universal property \cite[Proposition 10.3]{Deligne-interpolation} of the Deligne category we have a tensor functor $F_n = F_{n|n}: \underline{Rep}(GL_0) \to {\calR}_n$ mapping the standard representation of $\underline{Rep}(GL_0)$ to the standard representation of $GL(n\vert n)$ in ${\calR}_n$.  Every mixed tensor is in the image of this tensor functor \cite[8.13]{Comes-Wilson}. The indecomposable objects in $\underline{Rep}(GL_{\delta})$ are parametrized by bipartitions \cite{Comes-Wilson} and we denote by $R(\lambda)$ the indecomposable element associated to the bipartition $\lambda$. Then \cite[8.3]{Comes-Wilson} \[ F_n: R(\lambda) \mapsto \begin{cases} R_n(\lambda)  & \text{ if } k(\lambda) \leq n \\ 0 & \text{ if } k(\lambda) > n. \end{cases} \]  The atypicality of $R_n(\lambda)$ is given by $n - rk(\lambda)$ \cite{Heidersdorf-mixed-tensors}. Note that the superdimension of every nontrivial mixed tensor vanishes since $sdim(V) = 0$.

\section{The symmetric and alternating powers}\label{sec:symmetric-times-symmetric}

We define as in \cite{Heidersdorf-mixed-tensors} the following indecomposable modules in $\calR_n$ \[ \A_{S^i} = R(i;1^i) \text{ and } \A_{\Lambda^i}= (\A_{S^i})^{\vee} = R(1^i;i).\] The aim of this section is to prove prove a formula for $\A_{S^i} \otimes \A_{S^j}$ in $\mathcal{R_n}$ by calculations in $\underline{Rep}GL_0$. We neglect in this section summands that are not maximal atypical.

\subsection{Loewy layers and the $GL(1|1)$-case} Recall that we defined $S^i = [i,0,\ldots,0]$ for integers $i \geq 1$. We denote the trivial representation $S^0$ by $\one$. Furthermore we denote the projective cover of $[\lambda]$ by $P[\lambda]$.

\begin{lem}\cite[Lemma 13.3]{Heidersdorf-mixed-tensors} \label{lem:const} The Loewy structure of the $\A_{S^i}$ is given by ($n \geq 2$) \begin{align*} \A_{S^1} & = (\one, S^1, \one) \\ \A_{S^i} & = (S^{i-1}, S^i \oplus S^{i-2}, S^{i-1}) \quad 1<i \neq n \\ \A_{S^n} & = (S^{n-1}, S^{n} \oplus S^{n-2} \oplus B^{-1}, S^{n-1}). \end{align*} In particular $S^i$ is the constituent of highest weight in $\A_{S^i}$.
\end{lem}

We remark that mixed tensors are always rigid \cite[Corollary 5.4]{Heidersdorf-mixed-tensors}. These representations are maximal atypical for any $n$. We now derive a closed formula for the tensor products $\A_{S^i} \otimes \A_{S^j}$. It turns out that the maximal atypical summands are not irreducible whereas all other summands are irreducible. Therefore we split the computations in two parts: we first compute the projection to the maximal atypical block of $\A_{S^i} \otimes \A_{S^j}$ and deal with the remaining easier case later in section \ref{sec:gl(2|2)-lower-atypicality}. In the following formulas we often project to the maximal atypical block. Recall from \cite[Proposition 11.1]{Heidersdorf-mixed-tensors} that a mixed tensor $R(\lambda^L,\lambda^R)$ is maximal atypical if and only if $\lambda^R = (\lambda^L)^*$ where $\lambda^*$ denotes the conjugate partition. In this case we simply use the notation $R(\lambda^L)$, e.g. $\A_{S^i} = R(i)$ and $\A_{\Lambda^i} = R(1^i)$. 


\begin{lem} The atypical mixed tensors in $\calR_1$ are the $\A_{S^i}$ and their duals $\A_{\Lambda^j}$. They are the projective covers $\A_{S^i} = P[i-1]$ and $\A_{\Lambda^j} = P[-j+1]$.
\end{lem}

\begin{proof} Any mixed tensor $R(\lambda)$ in $\mathcal{R}_1$ satisfies $k(\lambda) \leq 1$. If $k(\lambda) = 0$, then $R(\lambda)$ is irreducible and singly atypical and therefore $R(\lambda) \cong \one$. If $k(\lambda)= 1$, then it is projective by theorem \ref{mixed-structure}. If $rk(\lambda) = 1$, $R(\lambda)$ is typical. The statement about the top follows from an explicit computation of the map $\theta: \Lambda \to X^+$ \cite[6.1]{Heidersdorf-mixed-tensors} already done in the proof of \cite[Lemma 13.3]{Heidersdorf-mixed-tensors}. 
\end{proof}

\begin{cor} In $\calR_1$ \begin{align*} \A_{S^i} \otimes \A_{\Lambda^j} = & \A_{S^{|-i+j|+2}} \oplus 2 \A_{S^{|-i+j|+1}} \oplus \A_{S^{|-i+j|}} \\
\A_{S^i} \otimes \A_{S^j} = & \A_{S^{i+j}} \oplus 2 \cdot \A_{S^{i+j-1}} \oplus \A_{S^{i+j-2}} \end{align*}
\end{cor}

\begin{proof} This is just rewriting the known formula ($a,b \in \Z$) \[ P[a] \otimes P[b] = P[a+b+1] \oplus 2 P[a+b] \oplus P[a+b-1] \] from \cite{Goetz-Quella-Schomerus}.
\end{proof}

Let us assume from now on that $n \geq 2$.

\begin{lem}\label{gl-1-1} After projection to the maximal atypical block ($n \geq 2$) \begin{align*} \A_{S^i} \otimes \A_{\Lambda^j} = & \A_{S^{|-i+j|+2}} \oplus 2 \A_{S^{|-i+j|+1}} \oplus \A_{S^{|-i+j|}} \oplus R_1 \\ \A_{S^i} \otimes \A_{S^j} = & \A_{S^{i+j}} \oplus 2 \cdot \A_{S^{i+j-1}} \oplus \A_{S^{i+j-2}} \oplus R_2 \end{align*} where $R_1$ and $R_2$ are direct sums of modules which do not contain any $\A_{S^i}$ or $\A_{\Lambda^j}$.
\end{lem}

\begin{proof} This follows from the $GL(1|1)$-case and the identification between the projective covers and the symmetric and alternating powers. In $GL(1|1)$ \cite{Goetz-Quella-Schomerus} \[P[a] \otimes P[b] = P[a+b-1] \oplus 2 P[a+b] \oplus P[a+b+1]. \] Hence this formula holds for the corresponding $\A_{S^i}$ respectively $\A_{\Lambda^j}$. It then holds in $\underline{Rep}(GL_0)$ up to summands in the kernel of the functor $F_{1}: \underline{Rep}(GL_0) \to Rep(GL(1|1))$ of section \ref{stable0}. The kernel consists of the $R(\lambda)$ with $k(\lambda) > 1$. By \cite[Lemma 13.1]{Heidersdorf-mixed-tensors} a maximal atypical mixed tensor satisfies $d(\lambda) = 1$ (and hence $k(\lambda) = 1$) if and only if and only if $\lambda = (i;1^i)$ or $\lambda = (1^i;i)$. Hence this formula holds in any $Rep(GL(n|n))$ up to contributions which lie in the kernel of $F_{n|n}: \underline{Rep}(GL_0) \to Rep(GL(n|n))$ and which are not $(1|1)$-cross. 
\end{proof}

\subsection{Tensor products in Deligne's category} In order to compute $\A_{S^i} \otimes \A_{S^j}$ in $\mathcal{R}_n$ we compute $R(i) \otimes R(j)$ in $\underline{Rep}(GL_0)$. We then push the result to $Rep(GL(n|n))$ using $F_{n}$. We recall the tensor product decomposition in $\underline{Rep}(GL_0)$.

\subsubsection{The lifting map} We attach to the weight diagram of a bipartition a cap-diagram as in \cite{Brundan-Stroppel-1} \cite{Comes-Wilson}. We denote the degree $\sum_i \lambda_i$ of a partition by $|\lambda|$. If $|\lambda| =n$ we write $\lambda \vdash n$. If $\lambda = (\lambda^L,\lambda^R)$ is a bipartition we denote its degree $(|\lambda^L|, |\lambda^R|)$ by $|\lambda|$ and we write $\lambda \vdash (r,s)$ if $|\lambda^L| = r$ and $|\lambda^R| = s$. Let us fix a bipartition $\lambda$ and consider the associated weight and cup diagram. For integers $i<j$ we say that $(i,j)$ is a $\vee \wedge$-pair if they are joined by a cap. For $\lambda, \mu \in \Lambda$ we say that $\mu$ is linked to $\lambda$ if there exists an integer $k \geq 0$ and bipartitions $\nu^{(n)}$ for $0 \leq n \leq k$ such that $\nu^{(0)} = \lambda,  \nu^{(k)} = \mu$ and the weight diagramm of $\nu^{(n)}$ is obtained from the one of $\nu^{(n-1)}$ by swapping the labels of some pair $\vee \wedge$-pair. Then we put \[ D_{\lambda, \mu} = \begin{cases} 1 \ \ \mu \text{ is linked to } \lambda \\ 0 \ \ \text{otherwise.} \end{cases} \] Then one has $D_{\lambda,\lambda} = 1$ for all  $\lambda$. Further $D_{\lambda,\mu} =  0$ unless $\mu = \lambda$ or $|\mu| = (|\lambda^L| - i, |\lambda^R| - i)$ for some $i> 0$ (see \cite[Theorem 6.2.3]{Comes-Wilson} for further details).  

\begin{example} We saw in example \ref{example-cup} that the cup diagram for $\lambda = ((4,2), (2^2, 1^{2}))$ is 

\medskip
\begin{center}
  \begin{tikzpicture}
 \draw (-6,0) -- (6,0);
\foreach \x in {} 
     \draw (\x-.1, .2) -- (\x,0) -- (\x +.1, .2);
\foreach \x in {} 
     \draw (\x-.1, -.2) -- (\x,0) -- (\x +.1, -.2);
\foreach \x in {} 
     \draw (\x-.1, .1) -- (\x +.1, -.1) (\x-.1, -.1) -- (\x +.1, .1);

\draw [-,black,out=270,in=270](3,0) to (4,0);
\draw [-,black,out=270,in=270](0,0) to (1,0);

\end{tikzpicture}
\end{center}
\medskip

Then there are 4 partitions linked to $\lambda$: $\lambda$ itself, 2 partitions obtained from $\lambda$ by interchanging the labels in one of the cups and a fourth summand by interchanging the labels in the two cups simultaneously. 
\end{example}

Let $t$ be an indeterminate and $R_{\delta}$ respective $R_{\delta,t}$ the Grothendieck rings of $\underline{Rep}(GL_{\delta})$ over $k$ respective of $\underline{Rep}(GL_t))$ over the fraction field $k((t-\delta))$. We follow the notation of \cite{Comes-Wilson} and denote by $(\lambda)$ or simply $\lambda$ the element $R(\lambda)$ in $R_{\delta,t}$ or $R_{\delta}$. Now define $\lift_{\delta}:R_{\delta} \to R_{\delta,t}$ as the $\Z$-linear map defined by $\lift_{\delta}(\lambda) = \sum_{\mu} D_{\lambda,\mu} \mu$ where the sum runs over all bipartitions $\mu$. By \cite[Theorem 6.2.3]{Comes-Wilson} $\lift_{\delta}$ is a ring isomorphism for every $\delta \in k$. 


\subsubsection{Generic tensor product decomposition} \label{sec:generic} By \cite[Theorem 7.1.1]{Comes-Wilson} the following decomposition holds for arbitrary bipartitions in $R_{\delta,t}$: \[ \lambda \mu = \sum_{v \in P \times P} \Gamma_{\lambda \mu}^{\nu} \nu\] with the numbers \[ \Gamma_{\lambda \mu}^{\nu} = \sum_{\alpha,\beta,\eta,\theta \in P} (\sum_{\kappa \in P} c_{\kappa \alpha}^{\lambda^L} c_{\kappa \beta}^{\mu^R}) \ (  \sum_{\gamma \in P} c_{\gamma \eta}^{\lambda^R} c_{\gamma \theta}^{\mu^L}) \ c_{\alpha \theta}^{\nu^L} c_{\beta \eta}^{\nu^R}, \] see \cite[Theorem 5.1.2]{Comes-Wilson}. Here $c_{\lambda \mu}^{\nu}$ denotes the Littlewood-Richardson coefficient and $P$ the set of all partitions. In particular if $\lambda \vdash (r,s)$, $\mu \vdash (r',s')$, then $\Gamma_{\lambda\mu}^{\nu} = 0$ unless $|\nu| \leq (r+r',s+s')$. So to decompose the tensor product $R(\lambda) \otimes R(\mu)$ in $\underline{Rep}(GL_{\delta})$ apply the following three steps: 

\medskip
\begin{enumerate}
\item Determine $\lift_{\delta}(\lambda \mu)$ in $R_{\delta,t}$, 
\item use the formula for $\Gamma_{\lambda \mu}^{\nu}$ above to compute the decomposition in $R_{\delta,t}$
\item and then take $\lift_{\delta}^{-1}$. 
\end{enumerate}

\subsection{Computations in $R_t$} \label{sec:comp} We continue to use our notation for the maximal atypical case and write $(i)$ instead of $(i;1^i)$. Clearly $\lift(i) = (i) + (i-1)$, $\lift(1^i) = (1^i) + (1^{i-1})$. We compute $R(i) \otimes R(j)$ in $\underline{Rep}(GL_0)$ following the three steps above. Hence in order to compute the tensor product $R(i) \otimes R(j)$ we have to compute the tensor product $(i) \otimes (j) + (i) \otimes (j-1) + (i-1) \otimes (j) + (i-1) \otimes (j-1)$ in $R_{\delta,t}$. We derive first a closed formula for $(i) \otimes (j)$ in $R_t$, i.e. for $((i,0,\ldots), (1^i)) \otimes (j,0,\ldots), (1^ {j})$. 

\begin{itemize}
 \item We analyze the sum $\sum_{\gamma \in P} c_{\gamma, \theta}^{\lambda^R} c_{\gamma,\eta}^{\mu^L}$. Here $\lambda^R= (1^i)$  and $\mu^L = (j,0,\ldots)$. We need to find all pairs of partitions $(a,b)$ such that $c_{a,b}^{\mu^L}$ is non-zero. We denote this by $(\mu^L)^{-1}$. Now the Pieri rule gives $(\mu^L)^{-1} = (0,j), (1,j-1),\ldots, (j-1,1), (j,0)$ and $(\lambda^R)^{-1} = (0, 1^i), (1,1^{i-1}),\ldots, (1^i,0)$. 
 Hence $c_{\alpha, \theta}^{\lambda^R} c_{\beta,\eta}^{\mu^L}$ is zero unless $(\gamma,\theta)$ and $(\gamma, \eta)$ are of the form $(0,i)$ and $(0, 1^{j})$ or are of the form $(1,i-1)$ and $(1,1^{j-1})$.
 

\item The contribution $\sum_{\kappa \in P} c_{\kappa,\alpha}^{\lambda^L} c_{\kappa,\beta}^{\mu^R}$: Here $\mu^R = (1^{j}), \ \lambda^L = (i)$. Similarly to the previous case this gives only the possibilities $c_{0,i}^{i} c_{0,1^{j}}^{1^{j}}$ and $c_{1,i-1}^{i} c_{1,1^{j-1}}^{1^{j}}$.
\end{itemize}


Hence the sum \[ \sum_{\alpha, \beta,\eta,\theta} ( \sum_{\kappa \in P} c_{\kappa,\alpha}^{\lambda^L} c_{\kappa,\beta}^{\mu^R})(\sum_{\gamma \in P} c_{\gamma, \eta}^{\lambda^R} c_{\gamma,\theta}^{\mu^L}) \] collapses to \[ (c_{0,i}^{i} c_{0,1^{j}}^{1^{j}} \ + \ c_{1,i-1}^{i} c_{1,1^{j-1}}^{1^{j}}) \ (c_{0,1^i}^{1^i} c_{0,j}^{j} + c_{1,1^{i-1}}^{1^i} c_{1,j-1}^{j}).  \] This corresponds to the choices 
\begin{itemize}
 \item (A) $\alpha = i, \ \beta= 1^{j}$
 \item (B) $\alpha = i-1, \ \beta =1^{j-1}$
 \item (C) $\eta =  1^i, \ \theta = j$
 \item (D) $\eta = 1^{i-1}, \ \theta = j-1$.
\end{itemize}

Only for these choices $AC, \ AD, \ BC, \ BD$ can there be a summand $(\nu)$ with nonvanishing $\Gamma_{\lambda\mu}^{\nu} = c_{\alpha,\theta}^{\nu^L} c_{\beta,\eta}^{\nu^R}$.

{\bf Notation:} From now on we only consider bipartitions $\nu$ with $\nu^L = (\nu^R)^{*}$ and think of such a bipartition as a partition $\nu^L$. Only these bipartitions will give maximal atypical summands in $\mathcal{R}_n$. The other summands can be easily calculated  later in section \ref{sec:gl(2|2)-lower-atypicality}.

\medskip
\begin{itemize}
 \item The AC-case: $c_{i,j}^{\nu^L} c_{1^{j},1^i}^{\nu^R} (\nu^L, \nu^R)$. By the Pieri rule $\nu^L$  can be any of $(i+j), \ (i+j -1,1), \ (i+j -2,2), \ldots$ and $\nu^R$ any of $(1^{i+ j})$,  $(2,1^{i+j-2}$, $\ldots$, $(i, |i-j|)$. Hence the following partitions $\nu$ (i.e. bipartitions of the form $(\nu^L; (\nu^L)^*$) appear with multiplicity 1: \[ (i+j), (i+j-1,1), \ldots, ((max(i,j),min(i,j)).\]

\item The AD-case: $c_{i,j-1}^{\nu^L} c_{1^{j},1^{i-1}}^{\nu^R}$. Restricting to $\nu^L = (\nu^R)^{*}$ we obtain \[ \nu \in \{ (i+j-1), (i+j-2,1), \ldots, ((max(i,j), min(i,j)-1)) \}.\] 

\item The BC-case: $c_{i-1,j}^{\nu^L} c_{1^{j-1}, 1^i}^{\nu^R}$. Here $\nu$ is any of \[ \nu \in \{ ((i+j-1), (i+j-2,1), \ldots, ((max(i,j), min(i,j)-1)) \} .\] 

\item The BD-case: $c_{i-1,j-1}^{\nu^L} c_{1^{j-1},1^{i-1}}^{\nu^R}$. Here \[ \nu \in \{ ((i+j-2), (i+j-3,1), \ldots, (max(i-1,j-1),min(i-1,j-1)).  \} \]

\end{itemize}

Hence in the Grothendieck ring $R_{\delta,t}$ \begin{align*}  & (i) \otimes (j)  = \\ & (i+j) + (i+j-1,1) + \ldots +  ((max(i,j),min(i,j)) \\  + & (i+j-1) +  (i+j-2,1) + \ldots + ((max(i,j), min(i,j)-1)) \\ + & (i+j-1) +  (i+j-2,1) + \ldots +  ((max(i,j), min(i,j)-1)) \\  + & ((i+j-2) +  (i+j-3,1) + \ldots +  (max(i-1,j-1),min(i-1,j-1)). \end{align*}



\subsection{Going back to $\underline{Rep}(GL_0)$} We calculate now the inverse $\lift^{-1}$ to get the decomposition in $\underline{Rep}(GL_0)$. In the special case $j = 1, i>1$ we get $(j-1) =0$ and hence $\lift((i) \otimes (1) ) = (i) \otimes (1) + (i) + (i-1) + (i-1) \otimes (1)$. In $R_t$ we have  \[ (i) \otimes (1) = (i+1) + (i,1) + 2(i) + (i-1)\] so that \[ \lift((i) \otimes (1)) = (i+1) + (i,1) + 4(i) + (i-1,1) + 4(i-1) + (i-2).\] After removing the contributions which will lead to $R(i+1) \oplus 2 R(i) \oplus R(i-1)$ we are left with $(i,1) + (i) + (i-1,1) + (i-1)$. This is the lift of $(i,1)$ and hence the indecomposable module $R(i,1)$ appears as a direct summand. 

\begin{lem} In $\underline{Rep}(GL_0)$ we have for $i \geq 2$
\[ R(i) \otimes R(1) = R(i+1) \oplus 2 R(i) \oplus R(i-1) \oplus R(i,1).\]
\end{lem}

In the general case we add up the contributions $((i) + (i-1)) \cdot ((j) + (j-1)) = (i)(j) + (i)(j-1) + (i-1)(j) + (i-1)(j-1)$. All the summands are of the following types $(a,0)$, $(a,b)$, $a> b>0$ or $(a,a),a>0$. We have \begin{align*} \lift(a,b) = & (a,b) + (a,b-1) + (a-1,b) + (a-1,b-1), \ \ \ a>b>0 \\ \lift(a,a) = & (a,a) + (a,a-1) + (a-1,a-2) + (a-2,a-2).\end{align*}

After removing the contributions in $R_{\delta,t}$ which will give the $R(i+j) \oplus 2 R(i+j-1) \oplus R(i+j-2)$ and applying successively the liftings from above we get the following decompositions. For $i >2, j=2$ we get \begin{align*} R(i) \otimes R(2) = & R(i+2) \oplus 2 R(i+2) \oplus R(i) \\ & \oplus R(i+1,1) \oplus R(i,2) \oplus 2 \cdot R(i,1) \oplus R(i-1,1) \end{align*} Assume now $i>2, \ j \geq 2$ and $i> j$. Then \begin{align*} R(i) \otimes R(j) = & R(i+j) \oplus 2 R(i+j-1) \oplus R(i+j-2) \\ & \oplus R(i+j-1,1) \\ & \oplus R(i+j-2,2) \oplus 2 \cdot R(i+j-2,1) \\ &  \oplus R(i+j-3,3) \oplus 2 \cdot R(i+j-3,2) \oplus R(i+j-3,1) \\ & \oplus R(i+j-4,4) \oplus 2 \cdot R(i+j-4,3) \oplus R(i+j-4,2) \\  & \oplus R(i+j-5,5) \oplus \ldots \\ & \oplus R(i, j) \oplus 2 \cdot R(i, j-1) \oplus R(i, j-2) \\ & \oplus R(i- 1, j-1). \end{align*} Now assume $i=j$. For $i=j=2$ we get \begin{align*} R(2) \otimes R(2) = & R(4) \oplus 2 R(3) \oplus R(2) \\ & \oplus R(3,1) \oplus R(2,2) \oplus 2 \cdot R(2,1).\end{align*} For $i = j > 2$ we get the same result as for $i \neq j$ while omitting the last factor $\oplus R(i-1, j-1)$.

\begin{remark} In the same way one can compute a closed formula of the tensor product $R(i) \otimes R(1^j)$. This is not needed for the $GL(2|2)$ calculations.
\end{remark}

\section{$GL(2|2)$ tensor products - the maximal atypical part}\label{sec:gl(2|2)}


We compute the decomposition of the tensor product of any two maximal atypical irreducible modules in ${\calR}_2$. In this section we compute only the direct summands which are maximal atypical. The remaining summands are computed in section \ref{sec:gl(2|2)-lower-atypicality}. The basic idea is to look at our formulas for $\A_{S^i} \otimes \A_{S^j}$ in the Grothendieck group and use these to compute the composition factors of $S^i \otimes S^j$ recursively starting with the obvious tensor product $S^i \otimes S^0$. We then determine the decomposition into indecomposable summands using results on cohomological tensor functors \cite{Heidersdorf-Weissauer-tensor} and case-by-case distinctions.

\subsection{The ${\calR}_2$-case: Setup} Recall from section \ref{sec:rep} the Berezin determinant $B=Ber = L(1,\ldots,1 \ | \ -1, \ldots, -1)$. Its tensor powers $Ber^i = L(i,\ldots,i \ | \ -i,\ldots,-i)$ are referred to as Berezin twists. In general \[ Ber^i \otimes L(\lambda_1,\ldots,\lambda_n \  | \ \lambda_{n+1}, \ldots, \lambda_{2n}) = L(\lambda_1 + i,\ldots,\lambda_n + i\  | \ \lambda_{n+1} - i, \ldots, \lambda_{2n} - i).\] Every maximally atypical irreducible representation $L(\lambda) = [\lambda_1,\lambda_2]$ (in the notation of section \ref{2}) is a Berezin twist of a representation of the form $S^i:= [i,0]$ for $i \in \N$. Since tensoring with $Ber$ is a flat functor, it is therefore enough to decompose the tensor product $S^i \otimes S^j$. The Ext-quiver of the maximal atypical block $\Gamma$ of ${\calR}_2$ can be easily determined from \cite{Brundan-Stroppel-2}. It has been worked out by \cite{Drouot}. For all irreducible modules in $\Gamma$ we have  $dim Ext^1(L(\lambda), L(\mu)) = dim Ext^1(L(\mu),L(\lambda)) = 0 \text{ or } 1$. The Ext-quiver can be picturised as follows where a line segment between two irreducible modules denotes a non-trivial extension class between these two modules and where an irreducible module $[x,y]$ is represented as a point in $\Z^2$. 
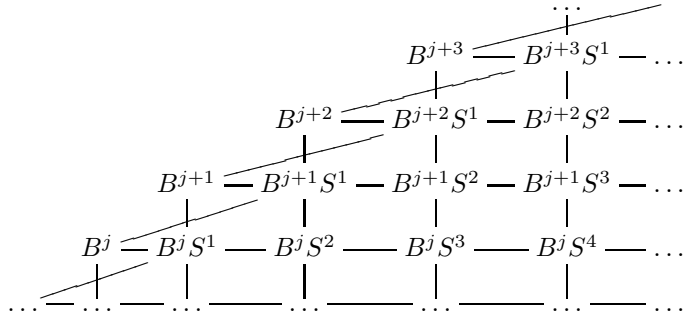
\[ \xymatrix@C=1em@R=1em{ & & & & & \ldots & \\ & & & & B^{j+3}  \ar@{-}[r] \ar@{-}[urr] & B^{j+3} S^1 \ar@{-}[u] \ar@{-}[r] & \ldots \\ & & & B^{j+2}  \ar@{-}[r] \ar@{-}[urr] & B^{j+2} S^1 \ar@{-}[u] \ar@{-}[r] & B^{j+2} S^2 \ar@{-}[u] \ar@{-}[r] & \ldots \\  & & B^{j+1}  \ar@{-}[r] \ar@{-}[urr] & B^{j+1} S^1 \ar@{-}[u] \ar@{-}
[r] & B^{j+1} S^2 \ar@{-}[u] \ar@{-}[r] & B^{
j+1} S^3 \ar@{-}[u] \ar@{-}[r] & \ldots \\ & B^j  \ar@{-}[r] \ar@{-}[urr] & B^j S^1 \ar@{-}[u] \ar@{-}[r] & B^j S^2 \ar@{-}[u] \ar@{-}[r] & B^j S^3 \ar@{-}[u] \ar@{-}[r] & B^j S^4 \ar@{-}[u] \ar@{-}[r] & \ldots \\ \ldots \ar@{-}[urr] \ar@{-}[r] & \ldots \ar@{-}[r] \ar@{-}[u] & \ldots \ar@{-}[r] \ar@{-}[u] & \ldots \ar@{-}[r] \ar@{-}[u]  & \ldots \ar@{-}[r] \ar@{-}[u]  & \ldots \ar@{-}[r] \ar@{-}[u]  & \ldots  } \]

The Loewy structure of the projective covers of a maximally atypical irreducible module can also be computed from \cite{Brundan-Stroppel-4} or be taken from Drouot: For $[a,b], a = b+k, k \geq 3$ the Loewy structure (we display the socle layers) is \[ P[a,b] = \begin{pmatrix} B^{a-k}S^k \\ B^{a-k}S^{k+1} \quad B^{a-k}S^{k-1} \quad B^{a-k-1} S^{k+1} \quad B^{a-k+1} S^{k-1} \\ 2 B^{a-k}S^k \quad B^{a-k-1}S^{k+2} \quad  B^{a-k-1}S^k \quad  B^{a-k+2} S^{k-3}  \\ B^{a-k}S^{k+1} \quad B^{a-k}S^{k-1} \quad B^{a-k-1} S^{k+1} \quad B^{a-k+1} S^{k-1} \\ B^{a-k}S^k \end{pmatrix}.\] For $[a,b], a = b+2$ the Loewy structure is \[ P[a,b] =  \begin{pmatrix} B^{a-2}S^2 \\ B^{a-2}S^3 \quad B^{a-2}S^1 \quad B^{a-3}S^3 \quad B^{a-1}S^1 \\ 2 B^{a-2}S^2 \quad B^{a-3}S^4 \quad B^{a-3}S^2 \quad B^{a-1}S^2 \quad B^{a-1} \quad B^{a-2} \\ B^{a-2}S^3 \quad B^{a-2}S^1 \quad B^{a-3}S^3 \quad B^{a-1}S^1  \\ B^{a-2}S^2 \end{pmatrix}.\] For $[a,b], a =b + 1$ the Loewy structure is \[ P[a,b] = \begin{pmatrix} B^{a-1}S^1 \\ B^{a-1}S^2 \quad B^{a-1} \quad B^{a-2} S^2 \quad B^a \quad B^{a-2} \\ 2 B^{a-1}S ^1 \quad B^{a-2}S^3 \quad B^{a-2}S^1 \quad B^a S^1  \\ B^{a-1}S^2 \quad B^{a-1} \quad B^{a-2} S^2 \quad B^a \quad B^{a-2} \\ B^{a-1}S^1 \end{pmatrix}.\] For $[a,b], a = b$ the Loewy structure is \[ P[a,b] = \begin{pmatrix} B^a \\ B^a S^1 \quad B^{a-1}S^1 \quad B^{a+1}S^1 \\ 2 B^a \quad  B^{a-1} \quad  B^{a-2} \quad  B^{a-1}S^2 \quad  B^a S^2 \quad  B^{a+1} \quad  B^{a+2} \\ B^a S^1 \quad  B^{a-1}S^1 \quad B^{a+1}S^1 \\ B^a \end{pmatrix}. \]



\subsection{The ${\calR}_2$-case: Mixed tensors}


All direct summands in the decomposition $R(i) \otimes R(j)$ in $\underline{Rep}GL_0$ satisfy $k(\lambda) \leq 2$. Hence they are not in the kernel of $F_{n|n}: \underline{Rep}(GL_0) \to \mathcal{R}_n$ for any $n \geq 2$. Therefore the formulas in the last section give us the maximal atypical summands in the decomposition of $\A_{S^i} \otimes \A_{S^j}$ for any $n \geq 2$. We specialise this decomposition to the ${\calR}_2$-case. All formulas hold only after projection to $\Gamma$.  It is easy to see that the $R(a,b)$ ($b > 0$) satisfy $k(\lambda) = 2$ and hence are projective covers of irreducible maximal atypical representations. The top and socle of these covers can be easily computed using the map $\theta: \Lambda \to X^+$ (see section \ref{sec:theta}). For small $j$ we get \begin{align*} \A_{S^1} \otimes \A_{S^1} = & \A_{S^2} \oplus 2 \cdot \A_{S^1} \oplus \A_{S^2}^{\vee} \\  \A_{S^i} \otimes \A_{S^1} = & \A_{S^{i+1}} \oplus 2 \cdot \A_{S^{i}} \oplus \A_{S^{i-1}} \oplus P[i-1,0]. \\
\A_{S^i} \otimes \A_{S^2} = & \A_{S^{i+2}} \oplus 2 \cdot \A_{S^{i+1}} \oplus \A_{S^{i}} \\ & \oplus P([i,0]) \oplus P([i-1,1] \oplus 2 \cdot P([i-1,0]) \oplus P([i-2,0]) \end{align*} where we assumed $i>1$ respectively $i>2$. Assume now $i>2, \ j \geq 2$ and $i> j$. \begin{align} \label{eq:A-TP}  \A_{S^i} \otimes \A_{S^j} = & \A_{S^{i+j}} \oplus 2 \cdot \A_{S^{i+j-1}} \oplus \A_{S^{i+j-2}} \nonumber \\ & \oplus P[i+j-2,0]) \\ & \oplus P[i+j-3,1] \oplus 2 \cdot P[i+j-3,0] \nonumber \\ &  \oplus P[i+j-4,2] \oplus 2 \cdot P[i+j-4,1] \oplus P[i+j-4,0] \nonumber \\ & \oplus P[i+j-5,3] \oplus 2 \cdot P[i+j-5,2] \oplus P[i+j-5,1] \nonumber \\  & \oplus P[i+j-6,4] \oplus \ldots \nonumber \\ & \oplus P[i-1, j-1] \oplus 2 \cdot P[i-1, j-2] \oplus P[i-1, j-3] \nonumber \\ & \oplus P[i - 2, j-2]. \nonumber \end{align} For $i=j=2$ \[ \A_{S^2} \otimes \A_{S^2} = \A_{S^4} \oplus 2 \A_{S^3} \oplus \A_{S^2} \oplus P[2,0] \oplus P[0,0] \oplus 2 P[1,0].\] For $i=j>2$ we have the same result without the last summand $P[i-2,j-2]$.



\subsection{The ${\calR}_2$-case: $K_0$-decomposition}

The tensor product decomposition of the $\A_{S^i} \otimes \A_{S^j}$ along with the knowledge of the composition factors of the indecomposable summands permits to give recursive formulas for the $K_0$-decomposition of the tensor products $S^i \otimes S^j$ in the Grothendieck ring $K_0 = K_0(\calR_n)$. Due to the asymmetry of the formulas and the asymmetry of the $K_0$-decompositions for $\A_{S^i}$ and $P[a,b]$ for small $i$ and $a-b$ we compute the tensor products for small $i$ and $j$ first. The $K_0$-decomposition $S^1 \otimes S^1$ follows immediately from the $\A_{S^1} \otimes \A_{S^1}$-decomposition and we get \[ S^1 \otimes S^1 = 2 {\bf 1} + 2 S^1 + B + B^{-1} + B^{-1} S^2 + S^2.\] Similarly one computes \begin{align*} S^2 \otimes S^1 & = 2 S^2 + S^3 + B^{-1}S^3 + S^1 + B S^1 \\ S^2 \otimes S^2 & = S^4 + B^{-1} S^4 +  2 S^3 + S^2 + BS^2 + 2 B S^1 + {\bf 1} + 2B + B^2.\end{align*} 

\begin{lem} We have $P[i,0] = 2 \A_{S^{i+1}} + B^{-1} \A_{S^{i+2}} + B \A_{S^i}$ for $i \geq 1$ in $K_0(\calR_2)$.
\end{lem}

\begin{proof} This is just a direct inspection of the Loewy structures above.
\end{proof}
                                     
\begin{lem}\label{composition-factors} For all $i > j$ we have in the Grothendieck group $K_0(\calR_2)$ \begin{align*} S^i & \otimes S^j =  2(S^{i+j-1} + Ber S^{i+j-3} + \cdots + Ber^{j-1} S^{i-j+1}) \\ & + S^{i+j}(1+Ber^{-1}) + S^{i+j-2}(1 + Ber^{-1}) + \cdots + Ber^j S^{i-j} (1+Ber^{-1}) \ .\end{align*} For $i=j$ we get \begin{align*} S^i \otimes S^i = & 2(S^{2i-1} + Ber S^{2i-3} + \cdots + Ber^{i-1} S^{1}) \\ &
+ S^{2i}(1+Ber^{-1}) + \cdots + Ber^i (1+Ber^{-1}) + B^{i-1} + B^{i-2} \ .\end{align*}
 
\end{lem}

\begin{proof}We first consider the cases $S^i \otimes S^1$ and $S^i \otimes S^2$ for $i>1$ respectively $i >2$.
The case $S^i \otimes S^1$, $i> 1$: For the induction start $i=2$ see above. Put $C_i = S^i \otimes S^1$ in $K_0({\calR}_n)$. For $i\geq 4$ we get then the uniform formula $S^i\otimes S^1 + 2 S^{i-1}\otimes S^1 + S^{i-2}\otimes S^1 =
(S^{i+1} + 2S^i + S^{i-1}) + (S^{i-1} + 2S^{i-2} + S^{i-3}) + (2C_{i-1} + Ber^{-1} S^{i+1}
+ Ber^{-1} S^{i-1} + Ber S^{i-1} + Ber S^{i-3})$. Hence using the induction assumption 
$S^{i-2}\otimes S^1 = 2S^{i-2} + S^{i-1} + Ber^{-1} S^{i-1} + S^{i-3} + Ber S^{i-3}$ we get
$S^i \otimes S^1 = 2S^i + S^{i+1} + S^{i-1} + Ber^{-1} S^{i+1} + Ber S^{i-1}$, and this proves the
induction step. Likewise for $S^i \otimes S^2$. Now assume $i>j>2$. Then for $\A_{S^i} \otimes \A_{S^j}$ we get using lemma \ref{lem:const} the regular formula in $K_0({\calR}_2)$ \begin{align*} \A_{S^i} \otimes \A_{S^j} = & S^i \otimes S^j + 4 (S^{i-1} \otimes S^{j-1}) + 2 (S^{i-1} \otimes S^j) + \\ & 2 (S^{i-1} \otimes S^{j-2}) + 2 (S^i \otimes S^{j-1}) + S^i \otimes S^{j-2} +  \\ & 2 ( S^{i-2} \otimes S^{j-1}) + S^{i-2} \otimes S^j + S^{i-2} \otimes S^{j-2}.\end{align*} All tensor products except $S^i \otimes S^j$ are known by induction. On the other hand this sum equals $\A_{S^{i+j}} + 2 \A_{S^{i+j}} + \A_{S^{i+j-2}} + P[i+j-2,0] + 2 P[i+j-3,0] + P[i+j-4,0] (1+ B)  +  2 B P[i+j-5,0] + B P[i+j-6,0] (1+B) + \ldots + 2 B^{j-2} P[i-j+1,0]  + B^{j-2} P[i-j,0] ( 1+ B)$.  Plugging in $P[a,0] = 2 \A_{S^{a+1}} + B^{-1} \A_{S^{a+2}} + \A_{S^a}$ for all $a \geq 1$ and comparing terms with the same $B$-power on both sides finishes the proof. The case $i=j$ works exactly the same way. \end{proof}


\subsection{The ${\calR}_2$-case: Socle Estimates}

We say $w(M)=k$ for a module $M$ if $M^\vee \cong Ber^{-k} M$.
Examples: $w(S^i)= i-1$ and $w(Ber) = 2$, and therefore
$$ w(S^i \otimes S^j) \ =\ i+j -2 \ .$$
On the other hand for $*$-selfdual modules $M$ we have 
$$ soc(M) \cong cosoc(M) \ ,$$ 
since $*$-duality is trivial on semisimple modules. On the other hand
$w(M)=k$ implies $soc(M)^\vee \cong Ber^{-k} cosoc(M)$, so that both
conditions together imply $w(soc(M))=k$. Hence being semi-simple, it is a direct sum of modules 
\[  soc(M) \cong soc'(M) \oplus \bigoplus_{\nu\in \bf Z} m(\nu) \cdot Ber^\nu S^{k+1-2\nu} \]
with $S^i = 0$ for $i <0$ and certain multiplicities $m(\nu)$, plus a sum $soc'(M)$ of modules of type
$$  \bigr( Ber^\nu \oplus Ber^{k-\nu-j+1}\bigl) S^j  $$
for certain $\nu\in \bf Z$ and certain natural numbers $j$ with $k-\nu - j+1 \neq \nu$. 
\bigskip

\begin{prop} For $n \geq 2$ and for $i> j\geq 2$ we have $soc'(M) =0$ for $M=S^{i-1} \otimes S^{j-1}$ and in $K_0(\calR_n)$ \[  soc(S^{i-1}\otimes S^{j-1}) \hookrightarrow 3 \cdot S^{i+j-3} + 2 \cdot Ber S^{i+j-5} + \cdots + 
2 \cdot Ber^{j-2} S^{i-j +1}  \ .\] For $i=j \geq 2$ we have in $K_0(\calR_n)$ \[  soc(S^{i-1}\otimes S^{i-1}) \hookrightarrow 3 \cdot S^{2i-3} + 2 \cdot Ber S^{2i-5} + \cdots +  
2 \cdot Ber^{i-2} S^{ 1} + B^{i-4}  \ .\]
\end{prop}

\begin{proof} Assume $i>j$. Note that $soc(M) \hookrightarrow soc(\A_{S^i} \otimes \A_{S^j})$ and by the formula \ref{eq:A-TP} from above
the latter is (using $soc(P[a,b]) = [a,b] = Ber^b S^{a-b}$) \begin{align*} S^{i+j-1} + & 3 S^{i+j-2} + 3 S^{i+j - 3} + (Ber + {\bf 1})S^{i+j-4} + 2 Ber S^{i+j-5} \\  + &  (Ber + {\bf 1})Ber S^{i+j-6} + 2 Ber^2 S^{i+j-7} + \cdots \\
+ &  (Ber + {\bf 1})Ber^{j-2} S^{i-j} + 2 Ber^{j-2} S^{i-j+1} \ .\end{align*}
Since $k=w(M) = (i-1) - 1 + (j-1) - 1 = i+j -4$, this implies the assertion $soc'(M)=0$. Indeed
the terms $  S^{i+j-1} +  3 S^{i+j-2}$ and also $N=(Ber + {\bf 1})Ber^\nu S^{i+j-4-2\nu}$ 
cannot contribute to $soc'(M)$, since \begin{align*}  N^\vee = & (Ber^{-1} + {\bf 1})Ber^{-\nu} Ber^{-i-j+3+2\nu} S^{i+j-4-2\nu} \\
= & (Ber^{-1} + {\bf 1})Ber^{-i-j+3+\nu} S^{i+j-4-2\nu}  \end{align*}
and \begin{align*} Ber^{-k}N = & Ber^{-k}(Ber + {\bf 1})Ber^\nu S^{i+j-4-2\nu} \\ = & (Ber^2 + Ber) Ber^{-i-j+3+\nu} S^{i+j-4-2\nu}\end{align*}
have no common irreducible summand. Hence 
$ soc(M)$ is contained in $ 3\cdot S^{i+j-3} + 2 \cdot Ber S^{i+j-5} + \cdots + 2 \cdot Ber^{j-2}S^{i-j+1}$. The proof is analogous for $i=j$. 
\end{proof}




\subsection{ The Duflo-Serganova functor $DS$}\label{DF}

We recall some constructions from the article \cite{Heidersdorf-Weissauer-tensor}.

\medskip {\it An embedding}. We view $G_{n-1}= GL(n-1|n-1)$ as an \lq{outer  block matrix}\rq\ in $G_n=GL(n|n)$ and $G_1$ as the \lq{inner  block matrix}\rq\  at the matrix positions $n\leq i,j \leq n+1$. Fix the following element $x\in \mathfrak{g}_1$, \[ x = \begin{pmatrix} 0 & y \\ 0 & 0 \end{pmatrix} \text{ for } \ y = \begin{pmatrix} 0 & 0 & \ldots & 0 \\ 0 & 0 & \ldots & 0 \\ \ldots & & \ldots &  \\ 1 & 0  & 0 & 0 \\ \end{pmatrix}. \] We furthermore fix the 
embedding \[ \varphi_{n,1}: G_{n-1} \times G_1 \hookrightarrow G_n \]
defined by $$\begin{pmatrix} A & B \\ C & D \end{pmatrix} \times \begin{pmatrix} a & b \\ c  & d\end{pmatrix} \mapsto \begin{pmatrix} A & 0 & 0 & B \\ 0 & a    & b    & 0\\ 0 & c   &  d     & 0\\ C & 0 & 0 & D \end{pmatrix}.$$ 
We use this embedding to identify elements in $G_{n-1}$ and $G_1$ with elements
in $G_n$.
In this sense $\epsilon_n = \epsilon_{n-1} \epsilon_1$ holds in $G_n$ (i.e. $\varphi_{n,1} (\epsilon_{n-1},\epsilon_1) = \epsilon_n$), for the corresponding elements $\epsilon_{n-1}$ and 
$\epsilon_1$ in $G_{n-1}$ resp. $G_1$, defined in section \ref{2}.

\medskip {\it Two functors}. 
One has a functor $(V,\rho) \mapsto V^+ =\{ v \in V\ \vert \ \rho(\epsilon_1)(v)=v \}$
$$ {}^+: {\calR}_n \to {\calR}_{n-1}$$ 
where $V^+$ is considered as a $G_{n-1}$-module using $\rho(\epsilon_1) \rho(g) = \rho(g) \rho(\epsilon_1)$
%
Similarly define $V^- =\{ v \in V\ \vert \ \rho(\epsilon_1)(v)=-v \}$. With the grading induced from $V=V_0 \oplus V_1$
this defines a representation $V^-$  of $G_{n-1}$  in $\Pi {\calR}_{n-1}$. Obviously
$$   (V,\rho)\vert_{G_{n-1}} \ =\ V^+  \ \oplus \ V^-  \ .$$  


\medskip {\it Cohomological tensor functors}.  Since $x$ is an odd element with $[x,x]=0$, we get $$2 \cdot \rho(x)^2 =[\rho(x),\rho(x)] =\rho([x,x]) =0 $$ for any representation
$(V,\rho)$ of $G_n$ in ${\calR}_n$. Notice $d= \rho(x)$ supercommutes with $\rho(G_{n-1})$. 
Furthermore $\rho(x): V^{\pm} \to V^{\mp}$ holds as a $k$-linear map, an immediate consequence of $d\rho(\varepsilon_{1}) = - \rho(\varepsilon_{1})d$, i.e. of $Ad(\varepsilon_1)(x)=-x$. 
Since $\rho(x)$ is an {\it odd} morphism,
$\rho(x)$ induces the following {\it even} morphisms (morphisms in ${\calR}_{n-1}$)
$$ \rho(x): V^+ \to \Pi(V^-) \quad {
and} \quad \rho(x): \Pi(V^-) \to V^+ \ .$$
The $k$-linear map $\partial=\rho(x): V\to V$ is a differential and commutes with the action of $G_{n-1}$ on $(V,\rho)$. Therefore $\partial$ defines a complex
in ${\calR}_{n-1}$ 
$$ \xymatrix{ \ar[r]^-{\partial} & V^+ \ar[r]^-{\partial} &  \Pi(V^-) \ar[r]^-{\partial} & V^+ \ar[r]^-{\partial} & \cdots } $$
Since this complex is periodic,  it has essentially only two cohomology groups denoted $H^+(V,\rho)$ and $H^-(V,\rho)$ in the following. This defines two functors $(V,\rho) \mapsto D_{n,n-1}^\pm(V,\rho)=H^{\pm}(V,\rho)$ 
$$  D_{n,n-1}^\pm: {\calR}_n \to {\calR}_{n-1}.$$ 


For the categories $\mathcal{T}=\mathcal{T}_n$ resp. $\mathcal{T}_{n-1}$ (for the groups
$G_n$ resp. $G_{n-1}$) consider the  
tensor functor of Duflo and Serganova in \cite{Duflo-Serganova} $$ DS_{n,n-1}: \mathcal{T}_n \to \mathcal{T}_{n-1} $$
defined by $DS_{n,n-1}(V,\rho)= V_x:=Kern(\rho(x))/Im(\rho(x))$. Then for $(V,\rho)\in \calR_n$  
$$   H^+(V,\rho) \oplus \Pi (H^-(V,\rho)) = DS_{n,n-1}(V) \ .$$
Indeed, the left side is $DS_{n,n-1}(V)=V_x$ for the $k$-linear 
map $\partial=\rho(x)$ on $V=V^+ \oplus V^-$. Hence $H^ +$  is the functor obtained by composing the tensor functor
$$ DS_{n,n-1}: {\calR}_n \to \mathcal{T}_{n-1} $$
with the functor 
$$ \mathcal{T}_{n-1} \to {\calR}_{n-1} $$
that projects the abelian category $\mathcal{T}_{n-1}$ onto ${\calR}_{n-1}$ using $\mathcal{T}_n = \calR_n \oplus \Pi \calR_n$.

%
%

\medskip
{\it The ring homomorphism $d$}. As an
element of the Grothendieck group $K_0(\calR_{n-1})$  we define for a module $M \in \calR_n$  $$d(M)= H^+(M) - H^-(M)\ .$$ The map
$d$ is additive by \cite{Heidersdorf-Weissauer-tensor}. 
Notice $$K_0(\mathcal{T}_n) = K_0(\calR_n) \oplus K_0(\calR_n[1]) = K_0(\calR_n) \otimes
(\mathbb Z \oplus \mathbb Z\cdot \Pi) \ .$$
We have a commutative diagram
$$  \xymatrix{  K_0(\mathcal{T}_n) \ar[d]_{DS} \ar[r] &    K_0(\calR_n) \ar[d]^d \cr
 K_0(\mathcal{T}_{n-1})  \ar[r] &    K_0(\calR_{n-1}) \cr} $$
 where the horizontal maps are surjective ring homomorphisms defined by
$\Pi \mapsto -1\ .$ 
Since $DS$ induces a ring homomorphism, $d$ defines
a ring homomorphism.



\subsection{The ${\calR}_2$-case: Indecomposability} If we display the maximal atypical composition factors $[x,y]$ of $S^i \otimes S^j$ in the $\Z^2$-lattice, we get the following picture. Here $\Box$ denotes composition factors occuring with multiplicity 2 and the $\circ$ appear with multiplicity 1. The socle is contained in the subset of composition factors denoted by $\Box$.


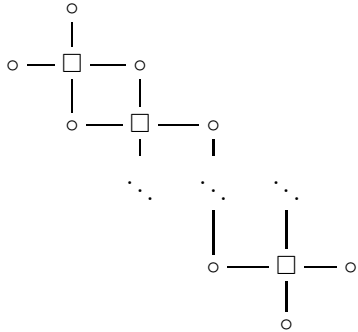
\[ \xymatrix@C=1em@R=1em{ & \circ \ar@{-}[d] & & & & & \\ \circ \ar@{-}[r] & \Box \ar@{-}[r] \ar@{-}[d] & \circ \ar@{-}[d] &  & & & \\ & \circ \ar@{-}[r]
& \Box \ar@{-}[d] \ar@{-}[r] & \circ \ar@{-}[d] & & &  \\  & & \ddots & \ddots & \ddots & \\   &
& & \circ \ar@{-}[r] \ar@{-}[u] & \Box \ar@{-}[r] \ar@{-}[d] \ar@{-}[u]  & \circ \\   &
& &  & \circ &     } \] with the two $\circ$ to the upper left at position $B^j S^{i-j}$ and $B^{j-1}S^{i-j}$ and the ones to the lower right at position $B^{-1}S^{i+j}$ and $S^{i+j}$. The picture in the $i=j$-case is similar

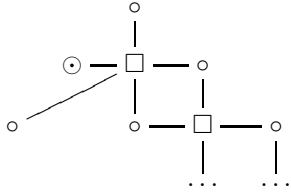
\[ \xymatrix@C=1em@R=1em{ & & \circ \ar@{-}[d] & & & &  \\ & \odot \ar@{-}[r] & \Box \ar@{-}[r] \ar@{-}[d] & \circ \ar@{-}[d] &  & &  \\ \circ \ar@{-}[urr] & & \circ \ar@{-}[r]
& \Box \ar@{-}[d] \ar@{-}[r] & \circ \ar@{-}[d] & &  \\ & & & \ldots & \ldots & }\] with the composition factor $\odot$ at position $B^{i-1}$ appearing with multiplicity 2 and the additional $\circ$ at position $B^{i-2}$.



\medskip

We now make use of the cohomological tensor functors $DS$. In the $GL(1|1)$-case $S^i \simeq B^i$ and hence $S^i \otimes S^j = S^{i+j}$. We know from \cite{Heidersdorf-Weissauer-tensor} that $DS(S^i) = S^i + \Pi^{1-i}B^{-1}$ and $DS(B) = \Pi^{-1}B$. Hence $DS(S^i \otimes S^j)$ splits into four indecomposable summands each of superdimension 1 or each of superdimension -1: \begin{align*} DS(S^i \otimes S^j) & = (S^i \oplus \Pi^{1-i} B^{-1}) \otimes (S^j \oplus \Pi^{1-j}B^{-1}) \\ & = B^{i+j} \oplus \Pi^{1-j}B^{i-1} \oplus \Pi^{1-i}B^{j-1} \oplus \Pi^{2-i-j}B^{-2}.\end{align*} Hence $M = S^i \otimes S^j$ splits into at most four indecomposable summands of $sdim \neq 0$.

\begin{lem} Every atypical direct summand is $^*$-invariant.
\end{lem}

\begin{proof} If $I$ is a direct summand which is not $^*$-invariant, $M$ contains $I^*$ as a direct summand and $[I]= [I^*]$ in $K_0({\calR}_n)$ since $^*$ identity on irreducible modules. However any summand of length $>1$ must contain a factor of type $\circ$ which occur in $M$ only with multiplicity 1, a contradiction. 
\end{proof}

\begin{cor} The superdimension of any maximally atypical summand is $\neq 0$.
\end{cor} 

\begin{proof} $M$ does not contain any projectice cover (look at composition factors). If $sdim(I) = 0$, $DS(I) = 0$. However $ker(DS) = \mathcal{C}^-$ \cite[Theorem 4.1]{Heidersdorf-Weissauer-tensor} (the modules with a filtration by anti-Kac-modules) which are not *-invariant, unless they are projective. 
\end{proof}

Assume $i>j$. By $*$-invariance the Loewy length of a direct summand is either 1 or 3. If $I$ is irreducible, then necessarily $I = \Box$ for a composition factor of the socle. By socle considerations both $\Box$ will split as direct summands. The remaining module would have  superdimension zero, hence the Loewy length of a direct summand is 3. Fix a composition factor of type $\Box$. The multiplicity of $\Box$ in the socle cannot be 2. If the multiplicity of $\Box$ in the socle is zero, then $\Box$ has to be in the middle Loewy layer. But this would force composition factors of type $\circ$ to be in the socle. Contradiction. Hence

\begin{cor} For $n \geq 2$ and $i> j$ \[ soc(S^{i}\otimes S^{j}) \ = \ S^{i+j-1} \oplus Ber S^{i+j-3} \oplus \cdots \oplus Ber^{j-1} S^{i-j +1}.\] 
\end{cor}  

We conclude that the superdimension of a direct summand is either 2 or 4. Hence $M$ is either indecomposable or splits into two summands $ M = I_1 \bigoplus I_2$ of superdimension 2. If $M$ would split, it would split in the following way: 

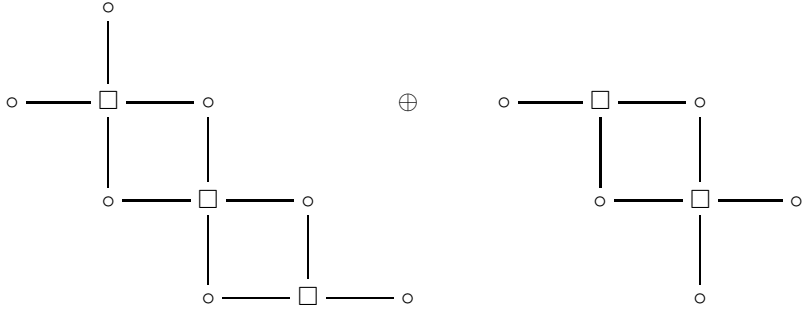
\[ \xymatrix{ & \circ \ar@{-}[d] & & &  \\ \circ \ar@{-}[r] & \Box \ar@{-}[r] \ar@{-}[d] & \circ \ar@{-}[d] &  & \oplus &\circ \ar@{-}[r] & \Box  \ar@{-}[r] \ar@{-}[d] & \circ \ar@{-}[d] & &  \\ & \circ \ar@{-}[r]
& \Box \ar@{-}[d] \ar@{-}[r] & \circ \ar@{-}[d] & & & \circ \ar@{-}[r] & \Box \ar@{-}[r]\ar@{-}[d]  & \circ & \\ &  & \circ \ar@{-}[u] \ar@{-}[r]& \Box  \ar@{-}[r]  & \circ  & &  & \circ &  }\]


Now we use the ring homomorphism $d: K_0(\calR_n) \to K_0(\calR_{n-1})$ defined by $d(M)= H^+(M) - H^-(M)$ as above. We know \[ d(S^i \otimes S^j) = B^{i+j} + (-1)^{1-j} B^{i-1} + (-1)^{1-i} B^{j-1} + (-1)^{2-i-j} B^{-2}\] since we can just take the formula for $DS(S^i \otimes S^j)$ and replace the parity shifts $\Pi^{i}$ by $(-1)^i$. Since $DS$ maps Anti-Kac modules to zero, $d$ applied to any square with edges $B^kS^i$, $B^{k+1}S^{i-1}$, $B^{k+1}S^i$, $B^k S^{i+1}$ is zero. Hence $d(I_2)$ is given by applying $d$ to the hook in the lower right $d(S^{i+j} + S^{i+j-1} + B^{-1}S^{i+j})$ and to $(B^v S^{i+j+1-2v} +  B^v S^{i+j-2v})$ from the upper left of $I_2$. We get $d(I_2) = B^{i+j} + (-1)^{i-j} B^{-2} + (-1)^v B^{i+j+1-v} + (-1)^v B^{i+j-v}$ with the two additional summands $(-1)^v B^{i+j+1-v} + (-1)^v B^{i+j-v}$. Contradiction, hence $M$ is indecomposable.

\medskip

Now assume $i=j$. By the socle estimates for $S^i \otimes S^i$ and $*$-duality either $B^{i-1}$ splits as a direct summand or both $B^{i-1}$ lie in the middle Loewy layer. Note that $Hom(B^{i-1}, S^i \otimes S^i) = Hom(B^{i-1} \otimes (S^i)^{\vee}, S^i) = End(S^i) =k$, hence the last case cannot happen. Hence $B^{i-1}$ splits as a direct summand. We show that the remaining module $M'$ in $S^i \otimes S^i = B^{i-1} \oplus M'$ is indecomposable. As in the $i>j$-case the Loewy length of any direct summand of $M'$ must be 3. As before we obtain for $i=j$ \[ soc(S^{i}\otimes S^{i}) \ = \ S^{2i-1} \oplus Ber S^{2i-3} \oplus \cdots \oplus Ber^{i-1} S^{1} \oplus B^{i-1}.\]  The remaining part $M'$ can either split into three indecomposable modules of superdimension one each, in a direct sum of two modules of superdimension one respectively two or is indecomposable. One cannot split the upper left  $\tilde{I}$  
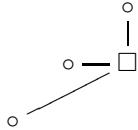
\[ \xymatrix@C=1em@R=1em{ & & \circ \ar@{-}[d] \\ & \circ \ar@{-}[r] & \Box \\ \circ \ar@{-}[urr] & &  }\] as a direct 
summand since its superdimension is $-1$. Similarly one cannot split  
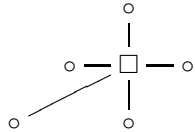
\[ \xymatrix@C=1em@R=1em{ & & \circ \ar@{-}[d] & \\ & \circ \ar@{-}[r] & \Box \ar@{-}[d] \ar@{-}[r] & \circ \\ \circ \ar@{-}[urr] & & \circ &  }\] as a direct summand since the remaining module would have superdimension zero. Since all composition factors except the $B$'s have superdimension $\pm 2$, $M'$ could split only into $M' = I_1 \oplus I_2$ with $sdim(I_1) = 1$ and $sdim(I_2) = 2$ with $I_2$ as above. We argue now as in the $i>j$-case. In the Grothendieck ring $K_0(\calR_n)$ \[ d(M) = B^{2i} + (-1)^{1-i} B^{i-1} + (-1)^{1-i} B^{i-1},\] but $d(I_2)$ has four summands as in the $i>j$-case. Contradiction, hence $M$ is indecomposable.


\begin{cor}\label{gl-2-2-final} Up to summands which are not in the maximal atypical block we obtain $S^i \otimes S^j \simeq M$ $(i > j$) where $M$ is
indecomposable with Loewy structure \[ \begin{pmatrix} S^{i+j-1} \quad Ber S^{i+j-3} \quad \cdots \quad Ber^{j-1} S^{i-j +1} \\  S^{i+j}(1+Ber^{-1}) \quad \cdots \quad Ber^j S^{i-j} (1+Ber^{-1})  \\ S^{i+j-1} \quad Ber S^{i+j-3} \quad \cdots \quad Ber^{j-1} S^{i-j +1}\end{pmatrix} \]
and $S^i \otimes S^i = B^{i-1} \oplus M$ where $M$ is indecomposable with Loewy
structure \[ \begin{pmatrix} S^{2i-1} \quad Ber S^{2i-3} \quad \cdots \quad Ber^{i-1} S^{1} \\ S^{2i}(1+Ber^{-1}) \quad \cdots \quad Ber^i S^{0} (1+Ber^{-1}) \quad B^{i-2}   \\ S^{2i-1} + Ber S^{2i-3} \quad  \cdots \quad Ber^{i-1} S^{1} \end{pmatrix}. \]
\end{cor}

We remark that the summand $Ber^{i-1}$ in $(\Pi^i S^i)^{\otimes 2}$ belongs to $\Lambda^2(\Pi^i S^i)$ and the summand $M$ to $Sym^2(\Pi^i S^i)$, see also \cite{Heidersdorf-Weissauer-Tannaka}. Note that $\Lambda^2(\Pi(V)) = Sym^2(V)$ for $V \in \calR_n$.

\subsection{The $GL(2|2)$-case} It is worth summarizing the situation in the $n=2$-case. In the $GL(2|2)$-case the irreducible representations are either typical, singly atypical or double (maximal) atypical. Every typical representation is a mixed tensor and every singly atypical irreducible representation is a Berezin twist of a mixed tensor. Hence the results of \cite{Heidersdorf-mixed-tensors} give the decomposition law for tensor products between typical and/or singly atypical irreducible representations. In \cite[Remark 13.4]{Heidersdorf-mixed-tensors} it is also explained how to decompose the tensor products between a typical and an irreducible maximal atypical representation in the $GL(2|2)$-case. Hence the fusion rules between irreducible representations are known except for the tensor product of a singly atypical and a maximal atypical representation, but these could be calculated by imitating our approach in the maximal atypical case. Since every irreducible maximal atypical representation of $GL(2|2)$ is of the form $[a,b]$ and any such representation is a Berezin twist of one of the $S^i = [i,0]$ for the Berezin determinant $Ber$, our result covers the entire maximal atypical $GL(2|2)$-case. For the $\mathfrak{psl}(2|2)$-case these decompositions were found prior by physicists \cite{Goetz-Quella-Schomerus-psl}.



\section{Reduction to the $GL(2|2)$-case} \label{reduction}

In this section we show that the determination of the maximal atypical summands in $S^i \otimes S^j$ in $\mathcal{R}_n$ is a corollary of the $n = 2$-case if we use the formalism of \cite{Heidersdorf-Weissauer-Tannaka}. Therefore this section depends unlike the other sections on some results of \cite{Heidersdorf-Weissauer-Tannaka}.

\subsection{Clean decomposition} We do not calculate the maximal atypical composition factors of $S^i \otimes S^j$ for $n \geq 3$. Nonetheless we can determine the number of indecomposable summands and their superdimension. We assume $n \geq 2$ and $i \geq j$.

\begin{lem} The Loewy length of a direct summand in $S^i \otimes S^j$ or $S^i \otimes (S^j)^{\vee}$ is $\leq 5$.
\end{lem}

\begin{proof} Since $S^i$ is in the socle and top of $\A_{S^{i+1}}$ we have a surjection $\A_{S^{i+1}} \otimes \A_{S^{j+1}} \to S^i \otimes S^j$. By the explicit formulas for $\A_{S^{i+1}} \otimes \A_{S^{j+1}}$ , the maximal Loewy length of a summand in  $\A_{S^{i+1}} \otimes \A_{S^{j+1}}$ is $\leq 5$. For that recall that the Loewy length of a mixed tensor $R(\lambda)$ equals $2d(\lambda) + 1$, and it is easy to check that $(a,b)$ satisfies $d(a,b) = 2$. Hence the quotient $S^i \otimes S^j$ has Loewy length at most 5. The case $S^i \otimes (S^j)^{\vee}$ is proved in the same manner. 
\end{proof}

Since the Loewy length of a maximal atypical projective cover in $\calR_n$ is $2n+1$ by \cite[Theorem 5.1]{Brundan-Stroppel-1} we get

\begin{cor} For all $n$ no maximal atypical projective cover appears in the decompositions $S^i \otimes S^j$ resp. $S^i \otimes (S^j)^{\vee}$. \label{thm:proj-free}
\end{cor}

\begin{proof} For $n=2$ we saw this by brute force computations. For $n \geq 3$ we have $2n+1 >5$.
\end{proof}

We say a direct sum is {\it clean} if none of the summands is negligible (i.e. has superdimension $0$). We say a negligible module in $\calR_n$ is {potentially projective of degree $r$} if $DS^{n-r}(N) \in \mathcal{T}_r$ is projective and $DS^{i}(N)$ is not for $i \leq n-r$. 

\begin{lem} Every maximal atypical negligible summand in a tensor product $L(\lambda) \otimes L(\mu)$ is potentially projective of degree at least 3.
\end{lem}

We proved in \cite[Corollary 5.8]{Heidersdorf-Weissauer-Tannaka} that the kernel of $DS$ equals $Proj$ if we restrict $DS$ to the full subcategory $\mathcal{T}_n^{\pm}$ of indecomposable modules occuring as direct summands in an iterated tensor product of irreducible modules.

\begin{proof} The maximal atypical part of the decomposition of $S^i \otimes S^j$ in $\mathcal{R}_2$ is clean by corollary \ref{gl-2-2-final}. Further $DS$ sends negligible modules in $\mathcal{T}_n^{\pm}$ to negligible modules in $\mathcal{T}_{n-1}^{\pm}$ \cite[Ccorollary 5.5]{Heidersdorf-Weissauer-Tannaka} and the kernel of $DS$ on $\mathcal{T}_n^{\pm}$ consists of the projective elements. Since $DS^{n-2} (L(\lambda) \otimes L(\mu)) \in \mathcal{T}_2$ splits into a direct sum of irreducible representations of the form $B^{a} S^{b}$ for some $a,b \in \Z$ by our $GL(2|2)$-computations, $DS^{n-2}(N) = 0$.
\end{proof}

\begin{lem} \label{thm:clean} For all $n$ the projection of $S^i \otimes S^j$ or $S^i \otimes (S^j)^{\vee}$ on the maximal atypical block is clean.
\end{lem} 

\begin{proof} We know that this is true for $n=2$. If $N$ is a maximal atypical summand in $S^i \otimes S^j$, we apply $DS$ several times until $N$ becomes projective. Since $DS(S^i) = S^i$  for  $i < n-1$ and $DS(S^i) = S^i \oplus \Pi^{n-1-i} Ber^{-1}$ for $i \geq n-1$, the tensor product $DS \circ \ldots \circ DS (S^i \otimes S^j)$ splits into a tensor product of $S^i$'s and Berezin powers. The projective summand coming from $N$ gives now a contradiction to \ref{thm:proj-free}. In the $S^i \otimes (S^j)^{\vee}$-case we can argue in the same way using $DS( (S^j)^{\vee} ) = DS( S^j)^{\vee}$. 
\end{proof}

\subsection{Nonvanishing superdimension} In this part we refer extensively to results from \cite{Heidersdorf-Weissauer-Tannaka}. We conclude from the previous paragraph that all maximal atypical summands in $S^i \otimes S^j$ have non-vanishing superdimension. 
Hence the direct summands can be seen in the quotient $\calR_n/\mathcal{N}$ by the modules of superdimension 0. According to \cite{Heidersdorf-semisimple} \cite{Heidersdorf-Weissauer-Tannaka} the tensor subcategory generated by the image of an irreducible element $L$ in this quotient is of the form $Rep(H_L,\epsilon)$ for some algebraic supergroup $H_L$ and some twist $\epsilon$ as in section \ref{2}. We apply this to the representations $S^i$. By abuse of notation we denote the image of $S^i$ in the quotient still by $S^i$.  We show in \cite{Heidersdorf-Weissauer-Tannaka} that the connected derived group $(H_{S^i}^0)_{der}$ of $S^i$ always satisfies \[ (H_{S^i}^0)_{der} \simeq SL(i+1) \text{ for } i \leq n-2 \] and \[(H_{S^i}^0)_{der} \simeq SL(n) \text{ for } i \geq n-1.\] Furthermore the restriction of $S^i$ (seen as a representation of $H_{S^i}$) to $(H_{S^i})_{der}^0$ remains irreducible. By superdimension reasons this restriction corresponds to the standard representation of $SL(i+1)$ or $SL(n)$ respectively. The derived group of the group corresponding to the tensor category $<S^i,S^j>, \ j \neq i$ generated by $S^i$ and $S^j$ in the quotient is the direct product of the derived groups corresponding to $S^i$ and $S^j$. If $S^i \otimes S^j$ would decompose as $M_1 \oplus M_2$ (up to superdimension 0) , the restriction to the derived group \[ H_{der}^0 \simeq (H_{S^i}^0)_{der} \times (H_{S^j}^0)_{der}\] would give a decomposition $Res(M_1) \oplus Res(M_2)$ of the tensor product $Res(S^i) \otimes Res(S^j)$. But this tensor product is the external tensor product of the standard representation of the first factor with the standard representation of the second factor. Hence $S^i \otimes S^j$ is indecomposable for $i \neq j$ up to superdimension $0$. If $i = j$ then the tensor product $S^i \otimes S^i$ behaves up to summands of superdimension $0$ like the tensor product of the standard representation of $SL(n)$ with itself. Since this tensor product splits into the two irreducible representations of weight $(2,0,\ldots,0)$ and $(1,1,0,\ldots,0)$, $S^i \otimes S^i$ has two indecomposable summands of non-vanishing superdimension.

\begin{cor} \label{s-i-for-n} The tensor product $S^i \otimes S^j$ in $\calR_n$ has a single indecomposable maximal atypical summand for $i \neq j$ and decomposes in two indecomposable summands for $i = j$.
\end{cor}

\begin{remark} In other words, once the know the $GL(2|2)$-case we get corollary \ref{s-i-for-n} for free using the formalism of \cite{Heidersdorf-Weissauer-tensor} \cite{Heidersdorf-Weissauer-Tannaka}. Note that we need very little of the machinery in \cite{Heidersdorf-Weissauer-Tannaka} in the $S^i$-case since this can be treated in an adhoc way \cite[Section 9]{Heidersdorf-Weissauer-Tannaka}.
\end{remark}




\section{The lower atypical summands in $\calR_n$}\label{sec:gl(2|2)-lower-atypicality}

We compute the remaining direct summands of the tensor product $S^i \otimes S^j$ in ${\calR}_n$ for $n \geq 2$. These direct summands are all irreducible which will follow from the fact that all summands of atypicality $<n$ in an $\A_{S^i} \otimes \A_{S^j}$ tensor product are irreducible and have vanishing $Ext^1$ with each other.

\begin{lem} $\A_{S^i} \otimes \A_{S^j}$ is a direct sum of maximally atypical summands and irreducible representations of atypicality $n-2$ and likewise for $\A_{\Lambda^i} \otimes \A_{\Lambda^j}$. 
\end{lem}

\begin{proof} In the decomposition of $\lift( (i;1^i) \otimes (j; 1^j) )$ in $R_t$, the bipartitions which will not contribute to the maximal atypical block are by the formulas in section \ref{sec:comp} of the form \[ ( \ (i+j-k,k);(2^r, 1^{i+j-2r}) \ )\] for some $k,r \geq 0$ and $k \neq r$ (note that the $k=r$ part would include the maximal atypical contributions that we already calculated).  We have \begin{align*} I_{\wedge} = & \{ i+j-k,k-1,-2,-3,-4,\ldots \} \\ I_{\vee} = & \{ -1,0,1,\ldots,r-2, r, r+1, \ldots,i+j-r-1, i+j-r+1, \ldots \} \end{align*} Since $k \neq r$, neither one of the two conditions $i+j-k = i+ j-r$, $k-1 = r-1$ is satisfied, hence the two sets intersect at two points, hence the weight diagram of any such bipartition has two crosses and two circles. Clearly the weight diagrams do not have any $\vee\wedge$-pair, hence the corresponding modules are irreducible. 
\end{proof}


\begin{lem} The composition factors of $S^i \otimes S^j$ in $\calR_n$ which are not maximally atypical are given by the set 
 \[ R((i+j-k,k);(2^r, 1^{i+j-2r})), \ k,r=0,1,\ldots,min(i,j), \ k \neq r.\] All these modules have atypicality $n-2$ and are irreducible.
\end{lem}

\begin{proof} This is again a recursive determination from the $\A_{S^i} \otimes \A_{S^j}$ tensor products. As before the $S^i \otimes S^1$ and $S^i \otimes S^2$-cases for $i\geq 1$ respectively $i \geq 2$ should be treated separately. For $S^i \otimes S^j$, $i,j \geq 3$ we obtain the regular formulas \begin{align*} \A_{S^i} \otimes \A_{S^{j}} = &  (S^i + 2 S^{i-1} + S^{i-2}) \otimes (S^{j} + 2 S^{j-1} + S^{j-2}) \\ = & S^i \otimes S^{j} + \text{ lower terms } \end{align*} where the lower terms are known by induction. Recall from section \ref{sec:generic} that $\Gamma_{\lambda\mu}^{\nu} = 0$ unless $|\nu| \leq (r+r',s+s')$. In the $\A_{S^i} \otimes \A_{S^{j}}$ tensor product the $R(,)$'s from above can therefore not occur for degree reasons in any tensor product $\A_{S^p} \otimes \A_{S^q}$ for $p \leq i, \ q \leq j$ where either $p<i$ or $q<j$. Hence they cannot occur in any tensor product decomposition of any $S^p \otimes S^q$ for $p,q$ as above, hence they have to occur in the $S^i \otimes S^{j}$-decomposition. The number of these modules is $(min(i,j)^2 - min(i,j)$. Substracting the inductively known numbers of 
not maximally atypical contributions in $S^p \otimes S^q$ in the $\A_{S^i} \otimes \A_{S^{j}}$-tensor product from the number of all such contributions in $\A_{S^i} \otimes \A_{S^{j}}$ we get $min(i,j)^2 - min(i,j)$ remaining modules. Hence there are no other summands in $S^i \otimes S^{j}$. 
\end{proof}

\begin{lem} The irreducible representation $R((i+j-k,k);(2^r, 1^{i+j-2r}))$ is isomorphic to \[ L(i+j - k,k, 0,\ldots,0| 0,\ldots,0, -r, -i -j + r).\]
\end{lem}

\begin{proof}. Let $m$ denote the maximal coordinate of a cross or circle in the weight diagram of the bipartition. To obtain the weight diagram of the highest weight we have to switch all labels to the right of this coordinate as well as the first $M - n + 2$ labels to its left which are not labelled $\times$ or $\circ$ by the explicit description of $\theta$ in \cite[6.1]{Heidersdorf-mixed-tensors}. Since we have four symbols $\times$ and $\circ$ this amounts to switching all the labels at positions $\geq -1$ and $<M$ (all of them $\vee$'s) and the $n-2$ $\wedge$'s at positions $-2,\ldots, -n +1$ to $\vee$'s. The crosses are at the positions $i+j-k, k-1$ and the circles at the positions $i+j-r, r-1$. The result follows. 
\end{proof}

\begin{lem} The lower atypical direct summands of $S^i \otimes S^j$ in $\calR_n$ are given by the set 
 \[ R((i+j-k,k);(2^r, 1^{i+j-2r})), \ k,r=0,1,\ldots,min(i,j), \ k \neq r.\] 
\end{lem}

\begin{proof} For any irreducible mixed tensors $R(\lambda), R(\mu)$ we have \[Ext_1(R(\lambda),R(\mu)) = 0\] since every block contains a unique irreducible mixed tensor by \cite[Lemma 8.1]{Heidersdorf-mixed-tensors}.
\end{proof} 

For a maximally atypical weight $(\lambda_1,\ldots,\lambda_n| - \lambda_n, \ldots,-\lambda_1)$ denote by \[ L_0(\lambda_1,\ldots,\lambda_n) \boxtimes  L_0( - \lambda_n, \ldots,-\lambda_1)\] the underlying irreducible $GL(n) \times GL(n)$-module. Denote by $\pi$ the following additive map from irreducible $GL(n) \times GL(n)$ modules to irreducible $GL(n|n)$-modules: \begin{align*} \pi & ( (L_0(\lambda_1,\ldots,\lambda_n)  \boxtimes  L_0( \mu_1, \ldots, \mu_n) ) \\ & = \begin{cases} 0 \qquad & L(\lambda_1,\ldots,\lambda_n | \mu_1,\ldots, \mu_n) \in \Gamma \\ L(\lambda_1,\ldots,\lambda_n | \mu_1,\ldots, \mu_n) \quad & \text{else. } \end{cases}\end{align*}

\begin{cor}\label{not-max-atypical-summands} The not maximally atypical contributions to $S^i \otimes S^j$ are given by \[ \pi ( \ ( L_0 (i,0,\ldots,0) \boxtimes L_0(0,\ldots,0,-i) \ ) \otimes ( \  L_0 (j,0,\ldots,0) \boxtimes L_0(0,\ldots,0,-j) \ ).\]
\end{cor}

\begin{cor}\label{final-result} For $n=2$ the tensor product $S^i \otimes S^j$ ($i >j$) decomposes as \begin{align*} S^i & \otimes S^j \simeq \begin{pmatrix} S^{i+j-1} \quad Ber S^{i+j-3} \quad \cdots \quad Ber^{j-1} S^{i-j +1} \\  S^{i+j}(1+Ber^{-1}) \quad \cdots \quad Ber^j S^{i-j} (1+Ber^{-1})  \\ S^{i+j-1} \quad Ber S^{i+j-3} \quad \cdots \quad Ber^{j-1} S^{i-j +1}\end{pmatrix}  \\ \oplus & \pi (  ( L_0 (i,0,\ldots,0) \boxtimes L_0(0,\ldots,0,-i)  ) \otimes (   L_0 (j,0,\ldots,0) \boxtimes L_0(0,\ldots,0,-j)  ).  \end{align*} The tensor product $S^i \otimes S^i$ decomposes as \begin{align*} S^i & \otimes S^i \simeq B^{i-1} \oplus \begin{pmatrix} S^{2i-1} \quad Ber S^{2i-3} \quad \cdots \quad Ber^{i-1} S^{1} \\ S^{2i}(1+Ber^{-1}) \quad \cdots \quad Ber^i S^{0} (1+Ber^{-1}) \quad B^{i-2}   \\ S^{2i-1} + Ber S^{2i-3} \quad  \cdots \quad Ber^{i-1} S^{1} \end{pmatrix} \\ \oplus & \pi ( ( L_0 (i,0,\ldots,0) \boxtimes L_0(0,\ldots,0,-i)  ) \otimes (   L_0 (i,0,\ldots,0) \boxtimes L_0(0,\ldots,0,-i)  ).  \end{align*}
\end{cor}
%
%
%



\section{The $GL(3|3)$-case and a conjecture}

The method applied to compute the $S^i \otimes S^j$ tensor products in the $GL(2|2)$-case works in principal for arbitrary $n$. Note that the results on the $\A_{S^i} \otimes \A_{S^j}$ tensor products are valid for any $n$. Furthermore we determined the part of $S^i \otimes S^j$ which is not  maximal atypical for any $n \geq 2$, hence we restrict here to the maximal atypical part. The obstacle to use the method of the $\calR_2$-case effectively is that the composition factors of the modules $R(a,b)$ appearing in the $\A_{S^i} \otimes \A_{S^j}$-case are difficult to compute. Decomposing a few $R(a,b)$ for small $a$ and $b$ in the $n=3$-case and then computing the composition factors of the $S^i \otimes S^j$ tensor products recursively, we arrive at the following tensor products ($\Lambda^2 = (S^2)^{\vee}$). Here we always project to the maximal atypical block.

\[ S^1 \otimes S^1 \simeq \one \oplus \begin{pmatrix} S^1 \\ S^2  +\quad (S^2)^{\vee} \quad \one  \\ S^1 \end{pmatrix}\]  
\[ S^2 \otimes S^1 \simeq \begin{pmatrix} S^2 \\  S^3 \quad [2,1,0] \quad S^1 \quad B^{-1} \\ S^2 \end{pmatrix}\] 
\[ S^3 \otimes S^1 \simeq \begin{pmatrix} S^3 \\  S^4 \quad [3,1,0] \quad S^2 \\ S^3 \end{pmatrix}\] 
\begin{align*} S^2 \otimes S^2 & \simeq [1,1,0] \\ & \oplus  \begin{pmatrix} S^3 \quad [2,1,0] \\ S^4 \quad [3,1,0] \quad [2,2,0] \quad (S^2)^{\vee} \quad [2,-1,-1] \quad [0,-1,-1] \quad S^2 \quad \one \\ S^3 \quad [2,1,0] \end{pmatrix}\end{align*}
\[ S^3 \otimes S^2 \simeq \begin{pmatrix} S^4 \quad [3,1,0] \\ S^5 \quad [4,1,0] \quad [3,2,0] \quad S^3 \quad [2,1,0] \quad [3,-1,-1] \\ S^4 \quad [3,1,0] \end{pmatrix} \] 
\begin{align*} & S^3 \otimes S^3 \simeq [2,2,0] \oplus \\ & \begin{pmatrix} S^5 \ \ [4,1,0] \ \  [3,2,0] \\ S^6 \ \ [5,1,0] \ \  [4,2,0] \  \ B^{-1}S^5 \  \ [3,3,0] \ \ S^4 \  \ [3,1,0] \ \  [1,1,0] \ \ [2,2,0] \\ S^5 \ \  [4,1,0] \ \ [3,2,0] \end{pmatrix}\end{align*} 



\begin{conj} \label{conjecture-socle} After projection to the maximal atypical block, $S^i \otimes S^j = M$ is  indecomposable if $i\neq j$ (corollary \ref{s-i-for-n}). $S^i \otimes S^i$ splits as \[ [i-1,i-1,\ldots,i-1,0] \oplus M.\] The socle of $M$ is for $i \geq j$ \[ soc(M) = [i+j-1,0,\ldots,0] \oplus [i+j-2,1,0,\ldots,0] \oplus \ldots \oplus [i,j-1,\ldots,0]\] and $M$ has Loewy length 3.  
\end{conj}

Note that since $\A_{S^i} \to \A_{S^j} \twoheadrightarrow S^i \otimes S^j$ and the maximal Loewy length of a direct summand $R(a,b)$ in $\A_{S^i} \to \A_{S^j}$ is 5, the Loewy length of $M$ is at most 5.

\section{Conflict of interest statement} On behalf of all authors, the corresponding author states that there is no conflict of interest.








\end{document}